\documentclass[a4paper, 12pt]{article}
\usepackage[utf8]{inputenc}
\usepackage[top=2.5cm,bottom=3.0cm,left=2.5cm,right=2.5cm]{geometry}
\usepackage{amssymb,amsmath,amsfonts,amsthm}

\usepackage{graphicx}  

\usepackage{soul}

\usepackage[pagewise]{lineno} 

\newtheorem{theorem}{Theorem}[section]
\newtheorem{lemma}[theorem]{Lemma}
\newtheorem{proposition}[theorem]{Proposition}
\newtheorem{corollary}[theorem]{Corollary}
\theoremstyle{definition}
\newtheorem{definition}[theorem]{Definition}

\newtheorem{example}[theorem]{Example}

\newtheorem*{theorem*}{Theorem} 
\newtheorem*{definition*}{Definition}
\newtheorem*{prop*}{Proposition}
\newtheorem*{corollary*}{Corollary}
\newtheorem{claim}{Claim}[section]

\title{Shadowing and Stability of Non-Invertible $p$-adic Dynamics}
\author{D.A. Caprio, F. Lenarduzzi, A. Messaoudi,  I. Tsokanos}
\date{ }

\begin{document}

	\maketitle
	
	\abstract{
		The stability theory of compact metric spaces with positive topological dimension is a well-established area in Dynamical Systems. A central result, attributed to Walters,  connects the concepts of topological stability and the shadowing property in invertible dynamics. 
		In contrast, zero-dimensional stability theory is a developing field, with an analogue of Walters' theorem for Cantor spaces being fully established only in 2019 by Kawaguchi.
		
		In this paper, we investigate the shadowing and stability properties of non-invertible dynamics in zero-dimensional spaces, focusing on the $p$-adic integers $\mathbb{Z}_{p} $ and the $p$-adic numbers $\mathbb{Q}_{p}$, where $p \geq 2$ is a prime number. The main result provides sufficient conditions under which the following families of maps exhibit strong shadowing and stability properties: 1) $p$-adic dynamical systems that are right-invertible through contractions, and 2) left-invertible contractions. Consequently, new examples of stable $p$-adic dynamics are presented.
		
	}
	


	\section{Introduction}
	
	A discrete dynamical system is a pair $ \left(  X, f  \right) $ where $X$ is a metric space and $f: X \to X$ is a continuous map (called dynamic). 
	One of the key objectives in the theory of dynamical systems is to study the asymptotic behavior of the orbits of points. Specifically, the orbit of a point $x \in X$ is the set $ O(x)=\{f^n(x),   n \in  \mathbb{N} \}$, where  $f^n$ denotes the $n$-th iteration of $f$, and $\mathbb{N} $ represents the set of non-negative integers.

	\paragraph{}

	In this paper, we investigate the shadowing and stability properties of a class of non-invertible dynamics on the $p$-adic group $ \left(  \mathbb{X}_{p},  \left| \left|  \cdot  \right|  \right|_{p}  \right) $, where $p\ge 2$ is a prime number, and $\mathbb{X}_{p}$ refers to either the locally compact space of the $p$-adic numbers $\mathbb{Q}_{p}$ or to the compact space of the $p$-adic integers $\mathbb{Z}_{p}$. Here, $ \left| \left|  \cdot  \right|  \right|_{p}$ denotes the usual $p$-adic norm. 
	The study of dynamical systems over $\mathbb{Z}_{p}$ or $\mathbb{Q}_p$ is an emerging field (see, for example,  \cite{Bastos-Caprio-Messaudi_Shadowing_and_Stability_in_p-adic_Dynamics},  \cite{Bryk_Silva-Measurable_Dynamics_of_simple_p-adic_Polynomials}, \cite{Darji_Concalves_Sobottka-Shadowing_Finite_Order_Shifts_and_Ultrametric_Spaces},  \cite{Durand_Paccaut-Minimal_Polynomial_Dynamics_on_the_set_of_3-adic_Integers}, \cite{Fan_Li_Yao-Strict_Ergodicity_of_Affine_p-adic_dynamical_systems_on_Zp}). This research has applications in fields such as physics, cognitive science, and cryptography among others (see \cite{Anashin-The_non-Archimedean_theory_of_discrete_systems}, \cite{Khrennikov_Nilsson-p-adic_Deterministic_and_Random_Dynamics},  \cite{Krhenikov_Oleschko_Lopez-Applications_of_p-adic_numbers_from_physics_to_geology}).
	
	Throughout this paper, for a continuous bounded map $f: \mathbb{X}_{p} \to \mathbb{X}_{p}$, the supremum norm is denoted by $ \left| \left|  f  \right|  \right|_{ \infty } = \underset{x \in \mathbb{X}_{p}}{\sup}  \left| \left|  f \left(  x  \right)   \right|  \right|_{p}$. 
	For $ x \in \mathbb{X}_{p}$ and $ \rho > 0$, the ball $B \left(  x, \rho  \right) $, centered at $x$ with radius $\rho$, is taken with respect to the $p$-adic norm $ \left| \left|  \cdot  \right|  \right|_{p}$.
	
	The dynamical definitions and concepts presented below are formulated for the space $\mathbb{X}_{p}$, but they can be readily adapted to general metric spaces.

	\paragraph{}
	The stability theory for compact metric spaces with positive topological dimension is well established \cite{Aoki-Topological_Dynamics, Pilyugin-Shadowing_in_Dynamical_Systems}.  Its fundamental result describes the relation between stability and the pseudo-orbit tracing property \cite{Walters_On_the_Pseudo_Orbit_Tracing_Property_and_its_relationship_to_Stability}. However, a zero-dimensional stability theory, particularly for Cantor spaces, has only recently begun to develop (see \cite{Bastos-Caprio-Messaudi_Shadowing_and_Stability_in_p-adic_Dynamics}, \cite{Kawaguchi_Topological_Stability_and_Shadowing_of_Zero-dimensional_Dynamical_Systems}).

	\smallskip
	The concept of stability was introduced by Andronov \& Pontrjagin \cite{Andronov_Pontrjagin-Systemes_Grossiers} and plays a fundamental role in various branches of dynamical systems theory. 
	A homeomorphism (respectively, a continuous map) $ f: \mathbb{X}_{p} \to \mathbb{X}_{p}$ is said to be \textit{structurally stable} if there is $\delta > 0$ such that for all homeomorphisms (respectively, for all continuous maps) $g: \mathbb{X}_{p} \to \mathbb{X}_{p}$ satisfying $ \left| \left|  f - g  \right|  \right|_{ \infty } \le \delta$, the system $ \left(  \mathbb{X}_{p}, f  \right) $ is \textit{topologically conjugate} to the system $ \left(  \mathbb{X}_{p}, g  \right) $; that is, there exists a homeomorphism $h: \mathbb{X}_{p} \to \mathbb{X}_{p}$ such that $f \circ h= h \circ g$. 
	The map $f$ is \textit{topologically stable} if for every $ \epsilon  > 0$, there exists $\delta  > 0$ such that, whenever $g: \mathbb{X}_{p} \to \mathbb{X}_{p}$ is a continuous map satisfying $ \left| \left|  f - g  \right| \right|  \le \delta$, there exists a continuous map $h: \mathbb{X}_{p} \to \mathbb{X}_{p}$ such that $ \left| \left|  h - id  \right|  \right|_{ \infty } \le  \epsilon $ and $ f \circ h = h \circ g$, where $id: \mathbb{X}_{p} \to \mathbb{X}_{p}$ denote the identity map on $\mathbb{X}_{p}$. 
	If, in addition, the map $h$ can be chosen to be surjective (respectively, a homeomorphism), then $f$ is said to be \textit{strongly topologically stable} (respectively, \textit{strongly structurally stable}). 
	
	\smallskip
	
	The concept of the pseudo-orbit tracing property, also known as shadowing, originates from the works of Bowen \cite{Bowen-Omega-limit_Sets_for_Axiom-A_Diffeomorphisms} \& Sinai \cite{Sinai-Gibbs_Measures_in_Ergodic_Theory}. Given a map $f: \mathbb{X}_{p} \to \mathbb{X}_{p}$ and $\delta  > 0$, a sequence $ \left(  x_{n}   \right)_{n \in \mathbb{N} }$ is called a \textit{$\delta$-pseudo-orbit} of $f$ if $ \left| \left|  f \left(  x_{n}  \right)  - x_{n+1}  \right|  \right|_{p} \le \delta $ for all $ n \in \mathbb{N} $.   
	A continuous map $f: \mathbb{X}_{p} \to \mathbb{X}_{p} $ is said to exhibit \textit{shadowing} (or \textit{positive shadowing}) if for every $\epsilon > 0$, there exists $\delta > 0$ such 
	that every $\delta$-pseudo-orbit $ \left(  x_n   \right)_{n \in\mathbb{N} }$ of $f$ is \textit{$\epsilon$-shadowed} by a real orbit of $f$; that is, there exists 
	$x \in \mathbb{X}_{p}$ such that $  \left| \left|  x_{n} - f^{n} \left(  x  \right)   \right|  \right|_{p} \leq \epsilon $ for all  $n \in \mathbb{N} $. 
	In this context, shadowing and positive shadowing are equivalent properties for continuous maps. 
	For homeomorphisms $f: \mathbb{X}_{p} \to \mathbb{X}_{p}$, the term \textit{shadowing} is used when the index set $\mathbb{N} $ is replaced by $\mathbb{Z} $ in the above definition.

	In many applications, the concept of shadowing is closely related to that of expansivity. A continuous map (respectively, a homeomorphism) $f: \mathbb{X}_{p} \to \mathbb{X}_{p}$ is said to be \textit{expansive} if there exists a positive constant $c > 0$ such that, for all distinct $x,y \in \mathbb{X}_{p}$, there exists $n \in \mathbb{N} $ (respectively, $n \in \mathbb{Z} $ for a homeomorphism) such that $  \left| \left|  f^{n} \left(  x  \right)  -  f^{n} \left(  y  \right)   \right|  \right|_{p} \ge c $.

	\smallskip
	
	Walter \cite{Walters_On_the_Pseudo_Orbit_Tracing_Property_and_its_relationship_to_Stability} demonstrated that every shadowing and expansive homeomorphism on a compact metric space is topologically stable. Furthermore, if the underlying space has a positive topological dimension, the homeomorphism is strongly topologically stable. Conversely, every topologically stable homeomorphism on a compact metric space with positive topological dimension is shadowing (see \cite{Walters_On_the_Pseudo_Orbit_Tracing_Property_and_its_relationship_to_Stability} for dimensions greater than $1$ and \cite{Morimoto-Stochastically_Stable_Diffeomorphisms_and_Takens_Conjecture} for dimension $1$). Recently, Kawaguchi \cite{Kawaguchi_Topological_Stability_and_Shadowing_of_Zero-dimensional_Dynamical_Systems} proved that this result also holds for Cantor spaces.

	\paragraph{}
	
	Lipschitz structural stability is another important type of stability extensively studied in the literature. Given $\delta  >0$, we define  
	$$ Lip_{\delta} \left(  \mathbb{X}_{p}  \right)  := \quad  \left\{  \phi : \mathbb{X}_{p} \to \mathbb{X}_{p}:   Lip \left(  \phi  \right)  \le \delta \quad  
	\text{and} \quad  \left| \left|  \phi  \right|  \right|_{ \infty } \le \delta   \right\}  $$ 
	to be the set of \textit{$\delta$-Lipschitz maps} over $\mathbb{X}_{p}$, where  
	$$ Lip \left(  \phi  \right)  := \quad \sup_{x  \neq y} {  \left| \left|   \phi(x) - \phi(y)  \right|  \right|_{p} \over  \left| \left|  x  -y  \right|  \right|_{p} }  $$
	denotes the \textit{Lipschitz constant} of $\phi$. Notice that, for an arbitrary continuous map $\phi: \mathbb{X}_{p} \to \mathbb{X}_{p}$, the Lipschitz constant $Lip \left(  \phi  \right) $ may be infinite.

	A map $f: \mathbb{X}_{p} \to \mathbb{X}_{p} $ is said to be \textit{Lipschitz structurally stable} if there exists $\delta > 0$ sufficiently small such that, for every $\phi \in Lip_{\delta} \left(  \mathbb{X}_{p}  \right) $, $f$ is topologically conjugate to $f + \phi$. 
	Furthermore, $f$ is \textit{strongly Lipschitz structurally stable} if, for every $ \epsilon > 0$, there exists $ \delta  > 0$ such that for every $\phi \in Lip_{\delta} \left(  \mathbb{X}_{p}  \right) $, the maps $f$ and $f+ \phi$ are conjugated by a homeomorphism $h$ that is $ \epsilon $-close to the identity, i.e., $ \left| \left|  h - id  \right|  \right|_{ \infty } \le  \epsilon $.

	It was shown in \cite{Kingsbery_Levin_Preygel_Silva-Dynamics_of_the_p-adic_shift_and_applications} that the shift map $\sigma: \mathbb{Z}_{p} \to \mathbb{Z}_{p} $, given by the formula 
	\begin{equation}\label{Eq1_Shift Map Zp}  
		\sigma  \left(  \sum_{i=0}^{+\infty} a_i p^{i}  \right)    :=   \sum_{i=0}^{+\infty} a_{i+1} p^{i}, \quad \text{for all }   \sum_{i=0}^{+\infty} a_i p^{i} \in \mathbb{Z}_{p},   \text{where }  a_i \in \{0, \ldots, p-1\} ,
	\end{equation}
	is strongly Lipschitz structurally stable. 
	More generally, this result holds for every $(p^{-k}, p^m)$-locally scaling map \cite{Bastos-Caprio-Messaudi_Shadowing_and_Stability_in_p-adic_Dynamics}. A map $f: \mathbb{Z}_{p} \to \mathbb{Z}_{p}$ is \textit{$  \left(  p^{-k}, p^{m}  \right) $-locally scaling} if 
	$ \left| \left|  f(x) - f(y)  \right|  \right|_{p}  =  p^{m}   \left| \left|  x - y  \right|  \right|_{p} $ whenever $ \left| \left|  x- y  \right|  \right|_{p} \le p^{-k}$, where $1 \leq m \leq k$  are integer numbers. 
	This class of maps was introduced in \cite{Kingsbery_Levin_Preygel_Silva-Dynamics_of_the_p-adic_shift_and_applications} and generalizes the concept of the shift map $ \sigma $, as $ \sigma $ is a $ \left(  p^{-1}, p  \right) $-locally scaling map. Locally scaling maps enjoy rich dynamical properties; they are shadowing, expansive, and Lipschitz structurally stable \cite{Bastos-Caprio-Messaudi_Shadowing_and_Stability_in_p-adic_Dynamics}.
	
	Families of maps $f: \mathbb{Q}_p \to \mathbb{Q}_p$  with rich dynamical properties also exist in the space of $p$-adic numbers. For instance, every homeomorphic \textit{contraction} (i.e. $Lip \left(  f  \right)   < 1$) and \textit{dilation} (i.e. $Lip \left(  f^{-1} \right)  < 1$) is shadowing, expansive, and Lipschitz structurally stable \cite{Bastos-Caprio-Messaudi_Shadowing_and_Stability_in_p-adic_Dynamics}.

	\paragraph{}	
	Shadowing and stability properties of dynamical systems have also been studied in the context of linear dynamics, where $\mathcal{X}$ is a Banach space and $f: \mathcal{X} \to \mathcal{X}$ is a linear continuous map. For instance, see \cite{Bernardes_Messaoudi-Shadowing_and_Structural_Stability_for_Operators}, \cite{Bernarde_Cirilo_Darji_Messaoudi_Pujals-Expansivity_and_Shadowing_in_Linear_Dynamics}, and \cite{Mahler-An_Interpolation_Series_for_Continuous_Functions_of_a_p-adic_Variable}. However, the techniques employed in this context differ from those used in the present work.

	\paragraph{}
	
	Classic stability results typically focus on homeomorphic dynamics, as demonstrated by Walter's Theorem \cite{Walters_On_the_Pseudo_Orbit_Tracing_Property_and_its_relationship_to_Stability}. Consequently, much of the research has been devoted to relaxing the conditions imposed by Walter. For example, one key question is: under what conditions is a (left or right) invertible dynamic stable? 
	
	\medskip
	
	One of the main results of this paper (Theorem \ref{Theor1_Dilations in p-adics}) concerns continuous right-invertible maps $f:\mathbb{X}_{p} \to \mathbb{X}_{p} $. Specifically, we assume that $f$ has a family of contractions  $R_{i}:\mathbb{X}_p\rightarrow \mathbb{X}_p$, indexed by $i \in \mathcal{I}$, as its right inverses, such that the images $R_{i} \left(  \mathbb{X}_{p}  \right) $, for $i \in \mathcal{I}$, are open and partition the space $\mathbb{X}_{p}$. Under these conditions, the map $f$ is shadowing, strongly Lipschitz structurally stable, and topologically stable. 
	
	As a consequence, new examples of non-invertible stable $p$-adic dynamics are obtained, while many already known examples are recovered (e.g., the $ \left(  p^{-k}, p^{k}  \right) $-locally scaling maps in $\mathbb{Z}_{p}$ and homeomorphic dilations in $\mathbb{Q}_{p}$). Additionally, this result is fairly complete; a right inverse contraction $R$ does not necessarily imply shadowing and stability conclusions for $f$ (Theorem \ref{Theor1_Single Right Inverse does not suffice for shadowing}). 
	To provide a counterexample, Kawaguchi's result \cite[Theorem 1.1]{Kawaguchi_Topological_Stability_and_Shadowing_of_Zero-dimensional_Dynamical_Systems} is extended to the class of uniformly continuous, topologically stable maps in $\mathbb{X}_{p}$ that are nowhere locally constant (i.e., there is no open ball where $f$ is constant). It is shown that such maps are shadowing.

	Complementary to the result above, sufficient conditions are provided for left-invertible contractions $R: \mathbb{X}_{p} \to \mathbb{X}_{p}$ to exhibit Lipschitz structural stability. 
	Throughout this paper, the contraction $R$ is said to be \textit{scaling} if there exists a function $\kappa:  \left\{  p^{n}:   n \in \mathbb{Z}  \right\}  \to  \left\{  p^{n}:   n \in \mathbb{Z}  \right\} \cup  \left\{  0  \right\} $ such that 
	\begin{equation}\label{Eq_Scaling Maps}
		\left| \left|  R \left(  x  \right)  - R \left(  y  \right)   \right|  \right|_{p}   =   \kappa \left(   \left| \left|  x - y  \right|  \right|_{p}  \right)  , \quad \text{for all } x,y \in \mathbb{X}_{p} .
	\end{equation} 
	It is shown in Theorem \ref{Theor3_Contractions in p-adics} that every open bi-Lipschitz contraction on $\mathbb{Z}_{p}$ and every bi-Lipschitz scaling contraction on $\mathbb{Q}_{p}$ is strongly Lipschitz structurally stable. 
	Moreover, the bi-Lipschitz condition is essential for this stability result, as illustrated by examples of non–bi-Lipschitz contractions that fail to be Lipschitz structurally stable.

	\paragraph{}	
	The paper is structured as follows. In Section \ref{Sec_Preliminaries and Notations}, we present the necessary preliminaries and define the notations employed throughout the paper.  
	The main results are stated in Section \ref{Sec_Main Results}, while Section \ref{Sec_Proof of the main results} is dedicated to the proofs of these results.

	\section{Preliminaries and Notation}\label{Sec_Preliminaries and Notations}
	
	Let $p \ge 2$ be a prime number. The ring of the $p$-adic integers $\mathbb{Z}_{p}$ can be viewed as the set of sequences of the form $  \left(  a_{i}   \right)_{ i \in \mathbb{N} }$, where  $a_{i} \in  \left\{  0, 1, \dots, p-1  \right\} $ for all $i \in  \mathbb{N} $. 
	Each element of $ x \in \mathbb{Z}_{p}$ can be uniquely expressed as a series 
	\begin{equation}\label{Eq2_p-adic Integers}
		x = \sum_{i= 0}^{+ \infty }  a_{i}(x) p^{i}, \quad  \text{where }  a_{i}(x) \in  \left\{  0, \dots, p-1  \right\}    \text{ for } i \ge 0, 
	\end{equation}
	On $\mathbb{Z}_{p}$, we can define the absolute value $ \left| \left|  \cdot  \right|  \right|_{p}$ by
	\begin{equation*}
		\left| \left|  x  \right|  \right|_{p} :=   \begin{cases} 
			p^{-u_{p}(x)}, & \text{if }  x  \neq 0 \\ 
			0, & \text{if } x =0 ,
		\end{cases} 
	\end{equation*}
	where $u_{p} \left(  x  \right) :=  \min \left\{  i \in  \mathbb{N} :     a_{i} \left(  x  \right)   \neq 0  \right\} $ denotes the smallest index of the non-zero digits in the $p$-adic expansion of $x$. 
	The absolute value $ \left| \left|  \cdot  \right|  \right|_{p}$ induces a metric $d_{p}$ on $\mathbb{Z}_{p}$ given by $d_{p}(x,y) =  \left| \left|  x -y  \right|  \right|_{p}$ for all $x,y \in \mathbb{Z}_{p}$. It is straightforward to verify that $d_{p}$ is an ultrametric, since 
	\begin{equation}\label{Eq2_p-adic Norm Triangle Inequality}
		\left| \left|  x -y  \right|  \right|_{p}   \le   \max \left\{   \left| \left|  x  \right|  \right|_{p},  \left| \left|  y  \right|  \right|_{p}  \right\}  \quad \text{for all } x,y \in \mathbb{Z}_{p} .    
	\end{equation}
	Furthermore, it is established that $\mathbb{Z}_{p}$ is the completion of $\mathbb{Z} $ with respect to the metric $d_{p}$ and that $\mathbb{Z}_{p}$ is compact.

	\paragraph{}
	The field of fractions of $\mathbb{Z}_{p}$ is denoted by $\mathbb{Q}_{p}$, which is defined as  
	$$ \mathbb{Q}_{p}    =    \mathbb{Z}_{p} \left[  {1 \over p} \right]    =    \underset{i \ge 0}{\bigcup} p^{-i} \cdot \mathbb{Z}_{p} .$$
	Any element $x \in \mathbb{Q}_{p}\backslash  \left\{  0  \right\} $ can be uniquely written as 
	\begin{equation*}\label{Eq2_p-adic Numbers}
		x  =  \sum_{i = u_{p} \left(  x  \right)  }^{+ \infty }  a_{i}(x) p^{i},   \text{ where } a_{i}(x) \in  \left\{  0, \dots, p-1  \right\}  \text{ for } i \ge u_{p} \left(  x  \right)  \text{ and } a_{u_{p} \left(  x  \right)  }  \neq 0. 
	\end{equation*}
	
	We define the quantity $u_{p} \left(  x  \right) $, the absolute value $ \left| \left|  \cdot  \right|  \right|_{p}$ and the metric $d_{p} \left(  \cdot, \cdot  \right) $  on $\mathbb{Q}_{p}$ analogously to the case of $\mathbb{Z}_{p}$.  It is shown that $\mathbb{Q}_{p}$ is the completion of $\mathbb{Q} $ under the $p$-adic metric $d_{p}$. Moreover, $\mathbb{Q}_{p}$ is a locally compact set, and $\mathbb{Z}_{p}$ is the closed unit ball of $\mathbb{Q}_{p}$ centered at zero. 
	
	\paragraph{} 
	Throughout the paper, given a $p$-adic number $x \in \mathbb{Q}_{p}$, the quantities $ \left\lfloor  x   \right\rfloor_{p}$ and $ \left\{  x  \right\}_{p}$ represent the integer and fractional parts of $x$ in the $p$-adic sense, respectively.
	For more information on $p$-adic numbers, the reader is referred to \cite{Robert-A_Course_in_p-adic_Analysis}.

	\paragraph{}

	\section{Main Results}\label{Sec_Main Results}
	
	The first result of the paper is concerned with dynamics in $\mathbb{X}_{p}$ that admit contractions as right inverses. 
	
	\begin{theorem}\label{Theor1_Dilations in p-adics} 
		Let $f:\mathbb{X}_p \to \mathbb{X}_p$ be a continuous map, where $\mathbb{X}_{p} \in  \left\{  \mathbb{Z}_{p}, \mathbb{Q}_{p}  \right\} $. Assume that the map $f$ admits right inverses given by a family of contractions  $R_{i}:\mathbb{X}_p\rightarrow \mathbb{X}_p$, indexed by $i \in \mathcal{I}$, i.e., $f\circ R_{i} = id $ and  $ Lip  \left(  R_{i}  \right)  <1$, for each $i \in \mathcal{I}$.
		
		If 
		\begin{equation}\label{Eq1_Inverses Covering Property}
			\bigcup_{i \in \mathcal{I}} R_{i} \left(  \mathbb{X}_{p}  \right)    =   \mathbb{X}_{p} ,
		\end{equation}
		then the map $f$ has the shadowing property.
		Furthermore, if the images 
		\begin{equation}\label{Eq1_Inverses Topological Property}
			R_{i} \left(  \mathbb{X}_{p}  \right) ,   i \in \mathcal{I},   \text{ are open and pairwise disjoint},
		\end{equation}
		then the map $f$ is strongly Lipschitz structurally stable and topologically stable. 
	\end{theorem}

	Theorem \ref{Theor1_Dilations in p-adics} immediately applies to homeomorphic dilations of $\mathbb{Q}_{p}$, such as the shift map $S: \mathbb{Q}_{p} \to \mathbb{Q}_{p}$ given by $ S(x) :=   p^{-1} \cdot x $. 
	An example of a map on $\mathbb{Q}_{p}$ that is not a homeomorphism but satisfies both conditions \eqref{Eq1_Inverses Covering Property} \& \eqref{Eq1_Inverses Topological Property} is the map 
	$$ f \left(  x  \right)  :=     \left\{  p^{-1} x  \right\}{p}  +  \left\lfloor  p^{-2}  x   \right\rfloor_{p}    =   p^{-1} \sum_{i = u_{p} \left(  x  \right)  }^{0} a_{i} \left(  x  \right)  p^{i} + p^{-2} \sum_{i=2}^{+ \infty } a_{i} \left(  x  \right)  p^{i},   $$
	where $x  = \sum_{i = u_{p} \left(  x  \right) }^{+ \infty } a_{i} \left(  x  \right)  p^{i} \in \mathbb{Q}_{p}$. 
	In this case, the map $f$ admits as right inverses the contractions 
	\begin{equation}\label{Eq1_rho-open Contractions}
		R_{a} \left(  x  \right)  := p \cdot  \left\{  x  \right\}_{p} + ap + p^{2}\cdot  \left\lfloor  x   \right\rfloor_{p} , \quad \text{with } a \in  \left\{  0, \dots, p-1  \right\}  .
	\end{equation} 
	Examples of maps in $\mathbb{Z}_{p}$ that satisfy the assumptions of Theorem \ref{Theor1_Dilations in p-adics} can easily be found or constructed, e.g., the shift map $\sigma: \mathbb{Z}_{p} \to \mathbb{Z}_{p}$ defined in \eqref{Eq1_Shift Map Zp}. The following corollary generalizes the example of the shift map $\sigma : \mathbb{Z}_{p} \to \mathbb{Z}_{p}$ to the class of $ \left(  p^{-k}, p^{k}  \right) $-locally scaling maps. 
	
	\begin{corollary}\label{Cor1_Locally Scaling Maps}
		Let $f: \mathbb{Z}_{p} \to \mathbb{Z}_{p}$ be a $ \left(  p^{-k}, p^{k}  \right) $-locally scaling map, where $k\ge 1$ is a positive integer. Then, $f$ is shadowing, strongly Lipschitz structurally stable, and topologically stable.
		
	\end{corollary}

	The conclusion of Theorem \ref{Theor1_Dilations in p-adics} is optimal, as under its assumptions, one cannot infer strong topological stability or (strong) structural stability. Indeed, notice that the shift map $ \sigma : \mathbb{Z}_{p} \to \mathbb{Z}_{p}$ (defined in \eqref{Eq1_Shift Map Zp}) satisfies the assumptions of Theorem \ref{Theor1_Dilations in p-adics} but it is neither structurally stable nor strongly topologically stable \cite{Bastos-Caprio-Messaudi_Shadowing_and_Stability_in_p-adic_Dynamics}.

	\paragraph{}
	
	The covering condition \eqref{Eq1_Inverses Covering Property} is crucial for ensuring the shadowing and stability properties of the map $f: \mathbb{X}_{p} \to \mathbb{X}_{p}$.

	\begin{theorem}\label{Theor1_Single Right Inverse does not suffice for shadowing} 
		There exists a continuous map $f: \mathbb{X}_{p} \to \mathbb{X}_{p} $ which admits a contraction $R: \mathbb{X}_{p} \to \mathbb{X}_{p}$ as a right inverse, but $f$ is neither shadowing, nor Lipschitz structurally stable, nor topologically stable.  
	\end{theorem}

	\paragraph{}
	
	Kawaguchi \cite{Kawaguchi_Topological_Stability_and_Shadowing_of_Zero-dimensional_Dynamical_Systems} proved that every topologically stable homeomorphism of a Cantor space is shadowing. 
	The proof of Theorem \ref{Theor1_Single Right Inverse does not suffice for shadowing} utilizes the following result, which generalizes Kawaguchi's result to the class of nowhere locally constant topologically stable maps in $\mathbb{X}_{p} \in  \left\{  \mathbb{Z}_{p}, \mathbb{Q}_{p}  \right\} $. 
	
	\begin{definition}[Nowhere Locally Constant maps]\label{Def_Nowhere Locally Stable}
		Given a continuous map $f :\mathbb{X}_{p} \to \mathbb{X}_{p}$ and $x \in \mathbb{X}_{p}$, we say that $f$ is \textit{locally constant} at the point $x$ if there exists $  \epsilon  > 0$ and $c \in \mathbb{X}_{p}$ such that for every $z \in B(x,  \epsilon )$, the condition $f(z) = c$ holds. The map $f $ is called \textit{nowhere locally constant} if it has no locally constant points.    
	\end{definition}
	
	\begin{proposition}\label{Prop1_Topological Stability implies Shadowing}
		Let $f : \mathbb{X}_{p} \to \mathbb{X}_{p}$, where $\mathbb{X}_{p} \in  \left\{   \mathbb{Z}_{p}, \mathbb{Q}_{p}  \right\} $, be a uniformly continuous map. If $f$ is topologically stable and nowhere locally constant, then $f$ has the shadowing property.     
	\end{proposition}

	\paragraph{}
	
	As far as contractions in $\mathbb{X}_{p}$ are concerned, it is well-known that they satisfy the shadowing property. 
	The next result, which complements Theorem \ref{Theor1_Dilations in p-adics}, addresses the stability properties of left-invertible contractions in $\mathbb{X}_{p}$.

	\begin{theorem}\label{Theor3_Contractions in p-adics}  
		Let $R: \mathbb{X}_{p} \to \mathbb{X}_{p}$ be a contraction whose image $R \left(  \mathbb{X}_{p}  \right) $ is open and which satisfies the bi-Lipschitz inequality 
		\begin{equation}\label{Eq2_Bi-Lipschitz Contractions}
			c_{1}\cdot  \left| \left|  x -y  \right|  \right|_{p}   \le       \left| \left|  R(x) - R(y)  \right|  \right|_{p}   \le   c_{2} \cdot  \left| \left|  x -y  \right|  \right|_{p}, \quad \text{for every }  x,y \in \mathbb{X}_{p},  
		\end{equation}
		where $ 0  < c_{1} \le c_{2}  < 1 $ are positive constants.  Then, the following statements hold:
		\begin{enumerate}

			\item Assume that $\mathbb{X}_{p} = \mathbb{Z}_{p}$. Then $R$ is strongly Lipschitz structurally stable.

			\item Assume that $\mathbb{X}_{p} = \mathbb{Q}_{p}$. If $R$ is scaling, then $R$ is strongly Lipschitz structurally stable.

		\end{enumerate}
		
	\end{theorem}

	Natural examples of open bi-Lipschitz scaling contractions in $\mathbb{X}_{p}$ are provided by the class of affine contractions. Specifically, these are maps $R: \mathbb{X}_{p} \to \mathbb{X}_{p}$ of the form $R(x) = vx + w$, where $v, w \in \mathbb{X}_{p}$ and $ 0 <  \left| \left|  v  \right|  \right|_{p} < 1$. 
	As a consequence, by Theorem \ref{Theor3_Contractions in p-adics}, affine contractions are strongly Lipschitz structurally stable. This result recovers \cite[Theorem 30]{Bastos-Caprio-Messaudi_Shadowing_and_Stability_in_p-adic_Dynamics}. 
	
	When $\mathbb{X}_{p} = \mathbb{Q}_{p}$, affine contractions are homeomorphisms. Non-homeomorphic examples of bi-Lipschitz scaling contractions can also be constructed in a straightforward manner. 
	For example, define the map $R$ by the formula
	$R \left(  x  \right)  = u \cdot  \left\{  x  \right\}_{p} + v\cdot  \left\lfloor  x   \right\rfloor_{p} + w $, 
	where $ u,v,w \in \mathbb{X}_{p}$ such that $ 0  <  \left| \left|  v  \right|  \right|_{p}  <  \left| \left|  u  \right|  \right|_{p}  < 1$. 
	Such maps include, for instance, the maps $R_{a}: \mathbb{Q}_{p} \to \mathbb{Q}_{p}$ defined in equation \eqref{Eq1_rho-open Contractions}.

	\paragraph{}

	In view of Theorem \ref{Theor3_Contractions in p-adics}, it is worth emphasizing that, for a general contraction $R$, the properties of being bi-Lipschitz and having an open image are not equivalent. This distinction is illustrated by the following examples. Consider the left-invertible contractions $R_{1}: \mathbb{Q}_{p}  \to \mathbb{Q}_{p}$ and $R_{2}: \mathbb{X}_{p} \to \mathbb{X}_{p}$ defined respectively by 
	$$ R_{1} \left(  x  \right)  :=   \begin{cases}
		p^{2} x, & \text{if } u_{p} \left(  x  \right)  \ge 0,  \\ 
		p  \left\{  x  \right\}  + p^{-u_{p} \left(  x  \right)  +1}  \left\lfloor  x   \right\rfloor, & \text{if } u_{p} \left(  x  \right)  \le - 1, 
	\end{cases}$$
	and 
	\begin{align*}
		R_{2} \left(  x  \right)  :&= p \left(   \left\{  x  \right\}  + \sum_{i=0}^{t_{x}} a_{i} \left(  x  \right)  p^{i} +  \left(  p-1  \right)  p^{t_{x} +1} + \sum_{i = t_{x} + 1}^{+ \infty } a_{i} \left(  x  \right)  p^{i+1}  \right)    \\ 
		&= p \left\{  x  \right\}  + p  \sum_{i=0}^{t_{x}} a_{i} \left(  x  \right)  p^{i} +  \left(  p-1  \right)  p^{t_{x} +2} + p^{2} \sum_{i = t_{x} +1}^{+ \infty } a_{i} \left(  x  \right)  p^{i},
	\end{align*} 
	where $t_{x} : = \min \left\{  i \ge 0 :   a_{i} \left(  x  \right)  = p - 1 \right\}  \in \mathbb{N} \cup  \left\{  + \infty   \right\} $. 
	The following properties of the maps $R_{1}$ and $R_{2}$ are readily verified. 
	The  image $R_{1} \left(  \mathbb{Q}_{p}  \right) $ is open, whereas $R_{1}$ is neither bi-Lipschitz nor scaling, since for all $x \in \mathbb{X}_{p}$ the digits indexed from $1$ to $\max \left\{  1, - u_{p} \left(  x  \right)   \right\} $ vanish. By contrast, the map $R_{2}$ is bi-Lipschitz but neither scaling nor has an open image. 
	Indeed, whenever $t_{x} = + \infty $, the point $R_{2} \left(  x  \right) $ is not an interior point of $R_{2} \left(  \mathbb{X}_{p}  \right) $; more precisely, $R_{2} \left(  \mathbb{X}_{p}  \right) $ contains no ball of the form $B \left(  R \left(  x  \right) , p^{k}  \right) $ for any $k \ge 1$. 
	As a complementary example, consider the left-invertible contraction $R_{3}: \mathbb{X}_{p} \to \mathbb{X}_{p}$, defined by  
	$$ R_{3} \left(  x  \right)  :=   p  \left\{  x  \right\}  + p^{2} \sum_{i=0}^{+ \infty }  \left(   a_{2i + 1} \left(  x  \right)  p^{2i} + a_{2i} \left(  x \right)  p^{2i +1 }   \right)  .$$
	The map $R_{3}$ is bi-Lipschitz and has open image, but it fails to be a scaling contraction.

	\paragraph{}
	
	The following example shows that the stability statement of Theorem \ref{Theor3_Contractions in p-adics} fails when the map $R$ is neither bi-Lipschitz nor has an open image, even if it is left-invertible and scaling.

	\begin{example} 
		Let $R: \mathbb{X}_{p} \to \mathbb{X}_{p}$ be the contraction defined by  
		$$ R \left(  x  \right)  :=   p  \left\{  x  \right\}  + p \sum_{i=0}^{+ \infty }  a_{i} \left(  x  \right)  p^{2i}.  $$ 
		It is straightforward to verify that $R$ is a scaling contraction. 
		However, it is not open, since for every point in $R \left(  \mathbb{X}_{p}  \right) $ all digits with positive even index vanish. Moreover, $R$ does not satisfy the bi-Lipschitz inequality \eqref{Eq2_Bi-Lipschitz Contractions}. 
		Finally, the map $L : \mathbb{X}_{p} \to \mathbb{X}_{p}$, given by $L \left(  x  \right)  :=  \left\{  p^{-1} x  \right\}  + \sum_{i=0}^{+ \infty } a_{2i+1} \left(  x  \right)  p^{i}$, is a left inverse of $R$. 
		
		To demonstrate that the contraction $R$ is not Lipschitz structurally stable, fix $n \in \mathbb{N} $ and consider the $\delta_{n}$-Lipschitz map $\phi_{n}: \mathbb{X}_{p} \to \mathbb{X}_{p}$, with $\delta_{n} = p^{-n-1}$, defined by $\phi_{n} \left(  x  \right)  := - a_{n} \left(  x  \right)  p^{2n+1}$. 
		Although $R$ is injective, the perturbed map $T_{n}:= R + \phi_{n}$ fails to be injective. Consequently, the maps $R$ and $T_{n}$ cannot be topologically conjugate, which completes the argument.  
		
	\end{example}

	\section{Proof of the main results}\label{Sec_Proof of the main results}

	\subsection{ Proof of Theorems \ref{Theor1_Dilations in p-adics} and Corollary \ref{Cor1_Locally Scaling Maps}}\label{Sec_Proof of Theorem Right Inverse}

	For the proof of Theorem \ref{Theor1_Dilations in p-adics}, the following lemma is employed. It shows that, if a continuous mapping $f: \mathbb{X}_{p} \to \mathbb{X}_{p}$ admits right inverses given by contractions, then for sufficiently small $\delta > 0$ and any $\phi \in Lip_{\delta} \left(  \mathbb{X}_{p}  \right) $, the analytic properties of the right inverses of $f$ are retained and inherited by the right inverses of the perturbed map $g = f + \phi$.

	\begin{lemma}\label{Lem1_Right Inverse of Noises}
		Let $f:\mathbb{X}_p \to \mathbb{X}_p$, with $\mathbb{X}_{p} \in  \left\{  \mathbb{Z}_{p}, \mathbb{Q}_{p}  \right\} $, be a continuous map that admits a contraction $R:\mathbb{X}_p \to \mathbb{X}_p$ as a right inverse, i.e., $Lip \left(  R  \right)  <1$ and $f\circ R = id $.
		Given $ \delta \in  \left(  0,  Lip \left(  R  \right) ^{-1} - 1  \right)  $ and $\phi \in       Lip_{\delta} \left(  \mathbb{X}_{p}  \right) $, define $ g :=   f + \phi $. 
		Then, there exists a contraction $\Tilde{R} : \mathbb{X}_{p} \to \mathbb{X}_{p}$          such that: 
		\begin{equation}\label{Eq1_Right Inverse}
			\text{1)}   g \circ \Tilde{R}   =   id,  \quad \text{2)}    \Tilde{R}  \left(  \mathbb{X}_{p}  \right)    =   R \left(  \mathbb{X}_{p}  \right) ,   \text{ and 3)}   Lip \left(  \Tilde{R}  \right)    \le   {Lip \left(  R  \right)   \over 1 - \delta \cdot Lip \left(  R  \right)  }  <   1.
		\end{equation}
		
	\end{lemma}

	\begin{proof} 
		Fix $\mathbb{X}_{p} \in  \left\{  \mathbb{Z}_{p}, \mathbb{Q}_{p}  \right\} $ and let $f, R, \delta, \phi$, and $g$ as described in the statement. Define the continuous map $H: \mathbb{X}_{p} \to \mathbb{X}_{p}$ by 
		\begin{align*} 
			H(x) : &=   g\circ R(x)    =   id(x) + \phi \circ R(x)   .
		\end{align*}
		Observe that $\phi\circ R$ is bounded and $\delta$-Lipschitz. To prove the relations \eqref{Eq1_Right Inverse}, it suffices to show that the map $H: \mathbb{X}_{p} \to \mathbb{X}_{p}$ is a homeomorphism. Once this is established, we define $ \Tilde{R}:=  R \circ H^{-1} $, and from  $g \circ \Tilde{R} = H \circ H^{-1} = id$, it follows that $\Tilde{R}$ is a continuous right inverse of $g$. Furthermore, since $H$ is a homeomorphism, it follows that $\Tilde{R} \left(  \mathbb{X}_{p}  \right)  = R \left(  \mathbb{X}_{p}  \right) $. 
		
		The fact that $H$ is a homeomorphism follows from standard arguments; therefore, only the main points are sketched. Injectivity is straightforward to verify. To establish surjectivity, fix $z \in \mathbb{X}_{p}$ and define $\Phi_{z}(x) = z - \phi \circ R(x)$.  Since $\Phi_{z}$ is a contraction, the existence of a fixed point yields the desired conclusion. Finally, the continuity of $H^{-1}$ follows from the fact that $H$ is an isometry. Indeed, because the mapping $\phi \circ R$ is a contraction, one obtains 
		$$  \left| \left|  H \left(  x_{1}  \right)  - H \left(  x_{2}  \right)   \right|  \right|_{p}   =    \left| \left|  x_{1} - x_{2} + \phi\circ R \left(  x_{1}  \right)  - \phi\circ R \left(  x_{2}  \right)   \right|  \right|_{p}   =    \left| \left|  x_{1} - x_{2}  \right| \right|  .$$

		\paragraph{}
		To complete the proof of the lemma, we now show that for every  $ 0 <  \delta  < Lip \left(  R  \right) ^{-1} -1  $, the map $\Tilde{R}: \mathbb{X}_{p} \to \mathbb{X}_{p}$ is a contraction with $Lip \left(  \Tilde{R}  \right) $ as described in the statement. 
		For convenience, set $ \rho :=     Lip \left(  R  \right) ^{-1} - 1  > 0 $.
		Since $R$ is a right inverse of $f$ and  $ Lip \left(  R  \right)   < 1$, for any  $x_{1}, x_{2} \in R \left(  \mathbb{X}_{p}  \right) $, we have   
		$$ \left| \left|  f \left(  x_{2}  \right)  - f \left(  x_{1}  \right)   \right|  \right|_{p} \quad \ge \quad  (1 + \rho) \cdot  \left| \left|  x_{2} - x_{1}  \right|  \right|_{p} . $$
		Now, fix $ 0  < \delta  < \rho$, and let $y_{1}, y_{2} \in \mathbb{X}_{p}$ with $\Tilde{R} \left(  y_{1}  \right)  = x_{1}$ and $\Tilde{R} \left(  y_{2}  \right)  = x_{2}$.  
		We obtain 
		\begin{align*}
			\left| \left|  y_{2} - y_{1}  \right|  \right|_{p}   &=     \left| \left|  g \left(  x_{2}  \right)  - g \left(  x_{1}  \right)   \right|  \right|_{p}   \ge     \left| \left|  f \left(  x_{2}  \right)  - f \left(  x_{1}  \right)   \right|  \right|_{p} -  \left| \left|  \phi \left(  x_{2}  \right)   - \phi \left(  x_{1}  \right)   \right|  \right|_{p}  .
		\end{align*}
		Thus, we have $ \left| \left|  y_{2} - y_{1}  \right|  \right|_{p} \ge (1 + \rho -\delta)\cdot  \left| \left|  x_{2} - x_{1}  \right|  \right|_{p} $, which, in turn, implies 
		$$  \left| \left|  \Tilde{R} \left(  y_{2}  \right)  - \Tilde{R} \left(  y_{1}  \right)   \right|  \right|_{p} \quad \le \quad  {1 \over (1 + \rho - \delta)} \cdot  \left| \left|  y_{2} - y_{1}  \right| \right| , $$
		where $  \left(   1 + \rho - \delta  \right) ^{-1} =   Lip \left(  R  \right)  \cdot  \left(  1 - \delta \cdot Lip \left(  R  \right)   \right) ^{-1}   <   1  $.
		
		This completes the proof of the lemma.

	\end{proof}

	\paragraph{}

	\begin{proof}[Proof of Theorem \ref{Theor1_Dilations in p-adics}]
		
		Throughout this proof, the set of bounded sequences in $\mathbb{X}_{p}$ is denoted by    
		\begin{equation*}
			\ell^{ \infty } \left(   \mathbb{N} ,    \mathbb{X}_{p}  \right)  : = \quad  \left\{   \left(  x_{n}   \right)_{n \in \mathbb{N} } \in \mathbb{X}_{p}^{\mathbb{N} }  : \underset{n \in \mathbb{N} }{\sup}    \left| \left|  x_{n}  \right|  \right|_{p}    <   +\infty  \right\}  
		\end{equation*} 
		Endowed with the supremum norm $ \left| \left|  \cdot  \right|  \right|_{\ell^{ \infty } \left(   \mathbb{N} ,    \mathbb{X}_{p}  \right) }$, the space $ \left(  \ell^{ \infty } \left(   \mathbb{N} ,   \mathbb{X}_{p}  \right) ,  \left| \left|  \cdot  \right|  \right|_{\ell^{ \infty } \left(   \mathbb{N} ,   \mathbb{X}_{p}  \right) }  \right) $ is complete. Notice that $\ell^{ \infty } \left(   \mathbb{N} ,   \mathbb{Z}_{p}  \right) $ is the closed unit ball of the $\mathbb{Q}_{p}$-Banach space $\ell^{ \infty } \left(   \mathbb{N} ,   \mathbb{Q}_{p}  \right) $ and, in particular, $\ell^{ \infty } \left(   \mathbb{N} ,   \mathbb{Z}_{p}  \right)  = \mathbb{Z}_{p}^{\mathbb{N} }$.

		\paragraph{} 
		
		Fix $\mathbb{X}_{p} \in  \left\{  \mathbb{Z}_{p}, \mathbb{Q}_{p}  \right\} $ and let $f: \mathbb{X}_{p} \to \mathbb{X}_{p}$ be a continuous map admitting a family of contractions $R_{i}: \mathbb{X}_{p} \to \mathbb{X}_{p}$, for $i \in \mathcal{I}$, as right inverses that satisfy  the covering condition \eqref{Eq1_Inverses Covering Property}.

		\paragraph{Shadowing:} First we demonstrate that $f$ satisfies the shadowing property, namely, that every $\delta$-pseudo-orbit of $f$ is $\delta/2$-shadowed by $f$. 
		
		\paragraph{} To this end, let $\delta  >0$ and let $ \left(  x_{n}   \right)_{n \in \mathbb{N} }$ be a $\delta$-pseudo-orbit of $f$. Define the map $F : \ell^{ \infty } \left(   \mathbb{N} ,   \mathbb{X}_{p}  \right) \rightarrow \ell^{ \infty } \left(   \mathbb{N} ,   \mathbb{X}_{p}  \right) $ by
		\begin{equation}\label{Eq1_Fixed Point Operator}
			F \left(   \left(  u_{n}   \right)_{n \in \mathbb{N} }  \right)    =    \left(  R_{i_{n}}(x_{n+1} + u_{n+1}) - x_n    \right)_{n \in \mathbb{N} } , 
		\end{equation}
		where $i_{n}  \in \mathcal{I}$ is chosen such that $R_{i_{n}} \circ f  \left(  x_{n}  \right)  = x_{n}$. The existence of such an $i_{n} \in \mathcal{I}$ is guaranteed by condition \eqref{Eq1_Inverses Covering Property}. Specifically, for each $x \in \mathbb{X}_{p}$, there exists $i = i(x) \in \mathcal{I}$ and $y = y(x) \in \mathbb{X}_{p}$ such that $R_{i}(y) = x$, which implies that $R_{i} \circ f(x) = R_{i} \circ f \circ R_{i}(y) = R_{i}(y) = x $. 
		If there are multiple possible choices for $i_{n}$, we may select $i_{n}$ arbitrarily. 
		
		The map $F$ is well defined; given a bounded sequence $ \left(  u_{n}   \right)_{n \in \mathbb{N} }$ in $\mathbb{X}_{p}$, the sequence $ \left(   R_{i_{n}} \left(  x_{n+1} + u_{n+1}  \right)  - x_{n}   \right)_{n \in \mathbb{N} }$ is also bounded. Specifically, for each $n\in \mathbb{N} $, one has also the strong inequality 	
		\begin{align*}
			\left| \left|  R_{i_{n}} \left(  x_{n+1} + u_{n+1}  \right)  - x_{n}  \right|  \right|_{p}   &=     \left| \left|  R_{i_{n}} \left(  x_{n+1} + u_{n+1}  \right)  - R_{i_{n}} \left(  f \left(  x_{n}  \right)   \right)   \right|  \right|_{p}   \\   		
			& <    \max \left(   \left| \left|  u_{n+1}  \right|  \right|_{p} , \delta   \right) ,
		\end{align*}
		where the last relation follows from the fact that $R$ is a contraction, together with the strong triangle inequality. 
		
		\paragraph{}        
		The map $F: \ell^{ \infty } \left(   \mathbb{N} ,   \mathbb{X}_{p}  \right) \to \ell^{ \infty } \left(   \mathbb{N} ,   \mathbb{X}_{p}  \right) $ is a contraction as the maps $R_{i}$, $i\in \mathcal{I}$, are contractions. Therefore, by the Banach fixed-point theorem, there exists a unique fixed point $ \left(  z_{n}   \right)_{n \in \mathbb{N} } \in \ell^{ \infty } \left(   \mathbb{N} ,   \mathbb{X}_{p}  \right) $ such that  
		\begin{equation}\label{Eq1_Fixed Point}
			F \left(   \left(  z_{n}   \right)_{n \in \mathbb{N} }  \right)     =    \left(  z_{n}   \right)_{n \in \mathbb{N} }   \text{ and }        R_{i_{n}}  \left(  x_{n+1} + z_{n+1}  \right)    =    x_{n} + z_{n},   \text{for all}   n \in  \mathbb{N} .  
		\end{equation}
		The second relation implies that $ f \left(  x_{n} + z_{n}  \right)  = x_{n+1} + z_{n+1} $, for all $n \in \mathbb{N} $, which in turns gives  
		\begin{equation}\label{Eq1_Shadowing Point}
			f^{n} \left(  x_{0} + z_{0}  \right)    =   x_{n} + z_{n}, \quad \text{for every } n \in  \mathbb{N} . 
		\end{equation}
		
		To conclude the proof of the shadowing property of $f$, it suffices to show that the $\delta$-pseudo-orbit $ \left(  x_{n}   \right)_{n \in \mathbb{N} }$ is $\delta/2$-shadowed by the orbit $ \left(  f^{n} \left(  x_{0} + z_{0}  \right)    \right)_{n \in \mathbb{N} }$. 
		From equation \eqref{Eq1_Shadowing Point}, it is sufficient to show that 
		\begin{equation}\label{Eq1_Fixed Point Norm}
			\left| \left|  z_{n}  \right|  \right|_{p}   \le   \delta/2, \quad \text{for all }  n \in  \mathbb{N} . 
		\end{equation}
		For each $n \in \mathbb{N} $, observe that 
		\begin{align*}
			\left| \left|  z_{n}  \right|  \right|_{p}   & \underset{\eqref{Eq1_Fixed Point}}{=}    \left| \left|  R_{i_{n}} \left(  x_{n+1} +z_{n+1}  \right)  - x_{n}  \right|  \right|_{p}  \\ 
			& \underset{\eqref{Eq2_p-adic Norm Triangle Inequality}}{\le}   \max  \left\{    \left| \left|  R_{i_{n}} \left(  x_{n+1} + z_{n+1}  \right)  - R_{i_{n}} \left(  x_{n+1}  \right)   \right|  \right|_{p} ,  \left| \left|  R_{i_{n}} \left(  x_{n+1}  \right)  - R_{i_{n}}  \left(  f  \left(  x_{n}  \right)   \right)   \right|  \right|_{p}  \right\}   \\ 
			& \le   \max \left\{  { \left| \left|  z_{n+1} \right|  \right|_{p} \over 2} ,   {\delta \over 2}  \right\}  .
		\end{align*}
		By iterating the last inequality, one obtains $ \left| \left|  z_{n}  \right|  \right|_{p} \le \max \left\{   \left| \left|  z_{n+k}  \right|  \right|_{p}/2^{k}, \delta/2  \right\} $, for all $k \ge 1$. Finally, inequality \eqref{Eq1_Fixed Point Norm} follows from the boundedness of the sequence $ \left(   \left| \left|  z_{n}  \right|  \right|_{p}   \right)_{n \in \mathbb{N} }$. Notice that one can also conclude inequality \eqref{Eq1_Fixed Point Norm} by taking the supremum on the two sides of the last inequality.

		\paragraph{Strong Lipschitz Structural Stability:}
		
		Let us assume, for the remainder of the proof, that the images $R_{i} \left(  \mathbb{X}_{p}  \right) $, for $i \in \mathcal{I}$, satisfy the condition given in equation \eqref{Eq1_Inverses Topological Property}; that is, they are open and pairwise disjoint.
		
		Fix a value $0  < \delta  < r -1 $, where $r$ is defined as  
		\begin{equation}\label{Eq1_Minimum of Lipschitz Constants}
			r :=   \min  \left\{  { 1 \over Lip \left(  R_{i}  \right)  } :    i \in \mathcal{I}  \right\}    \ge    p .
		\end{equation}
		Let $\phi \in Lip_{\delta} \left(  \mathbb{X}_{p}  \right) $ be a $\delta$-Lipschitz map. Define $g = f + \phi$. Since the choice of $\delta  > 0$ and $\phi \in Lip_{\delta} \left(  \mathbb{X}_{p}  \right) $ is arbitrary, to prove that $f$ is strongly Lipschitz structurally stable, it suffices to show that $f$ and $g$ are topologically conjugated by a homeomorphism $h$ which is $\delta$-close to the identity map. 
		
		To proceed, we note that, due to the choice of $  \delta  \in  \left(  0,  r-1  \right) $, Lemma \ref{Lem1_Right Inverse of Noises} guarantees that for each $i\in \mathcal{I}$, there exists a right inverse $\tilde{R}_{i}$ of $g$ (corresponding to the right inverse $R_{i}$ of $f$). 
		Specifically, the maps $\Tilde{R}_{i}$, $i \in \mathcal{I}$ are contractions, and it holds that $R_{i} \left(  \mathbb{X}_{p}  \right)  = \Tilde{R}_{i} \left(  \mathbb{X}_{p}  \right) $ for all $i \in \mathcal{I}$. Consequently, the images $\tilde{R}_{i} \left(  \mathbb{X}_{p}  \right) $ for $i \in \mathcal{I}$ satisfy the conditions of equations \eqref{Eq1_Inverses Covering Property} and \eqref{Eq1_Inverses Topological Property}.

		\paragraph{}
		
		The next part of the proof, from this point until the statement of Claim \ref{Claim 1.2}, follows similarly to the proof of the shadowing property of $f$.  
		
		Given $x,y \in \mathbb{X}_{p}$, the sequence $  \left(  g^{n} \left(  x \right)    \right)_{n \in \mathbb{N} }$ forms a $\delta$-pseudo-orbit of $f$, and the sequence $  \left(   f^{n}(y)   \right)_{n \in \mathbb{N} }$ forms a $\delta$-pseudo-orbit of $g$.
		In analogy to the map $F$ defined in \eqref{Eq1_Fixed Point Operator}, we define the maps $F_{x}, G_{y}: \ell^{ \infty } \left(   \mathbb{N} ,   \mathbb{X}_{p}  \right) \to \ell^{ \infty } \left(   \mathbb{N} ,   \mathbb{X}_{p}  \right) $ as follows: 
		\begin{equation}\label{Eq1_Theor1_Sh Operator F}
			F_{x}  \left(   \left(  u_{n}   \right)_{n \in \mathbb{N} }  \right)  := \quad  \left(  R_{i_{n}} \left(  g^{n+1}(x) + u_{n+1}  \right)  - g^{n}(x)   \right)_{n \in \mathbb{N} } 
		\end{equation}
		and
		\begin{equation}\label{Eq1_Prop1_Sh Operator G}
			G_{y} \left(   \left(  u_{n}   \right)_{n \in \mathbb{N} }  \right)  := \quad  \left(  \Tilde{R}_{j_{n}} \left(  f^{n+1}(y) + u_{n+1}  \right)  - f^{n}(y)   \right)_{n \in \mathbb{N} } ,
		\end{equation} 
		where $i_{n} = i_{n}(x) \in \mathcal{I}$ and $j_{n} = j_{n}(y) \in \mathcal{I}$ are indices satisfying the relations $R_{i_{n}} \circ f \left(  g^{n}(x)  \right)  = g^{n}(x)$ and $\Tilde{R}_{j_{n}} \circ g  \left(  f^{n}(y)  \right)  = f^{n}(y)$.
		These indices exist due to the covering condition \eqref{Eq1_Inverses Covering Property}. We denote the sequences of indices as 
		\begin{equation}\label{Eq1_Indices Operator FG}
			\mathbf{i}_{x} := \quad  \left(  i_{n}(x)   \right)_{n \in \mathbb{N} } \quad \text{and}           \quad  \mathbf{j}_{y} := \quad  \left(  j_{n}(y)   \right)_{n \in \mathbb{N} }.
		\end{equation}
		The maps $F_{x}$ and $G_{y}$ are well-defined due to the properties \eqref{Eq1_Inverses Covering Property} and \eqref{Eq1_Inverses Topological Property}.
		Moreover, these maps are contractions, and their fixed points $  \left(  z_{n}(x)   \right)_{n \in \mathbb{N} } \in \ell^{ \infty } \left(   \mathbb{N} ,   \mathbb{X}_{p}  \right) $ and $  \left(  \Tilde{z}_{n}(y)   \right)_{n \in \mathbb{N} } \in \ell^{ \infty } \left(   \mathbb{N} ,   \mathbb{X}_{p}  \right) $ are unique. 
		We define the maps $h, \tilde{h}: \mathbb{X}_{p} \to \mathbb{X}_{p}$ by the formulas 
		\begin{equation}\label{Eq1_Conjugation Lipschitz Structurally Stable FG}
			h(x)  :=   x + z_{0}(x) \quad \text{and} \quad  \Tilde{h}(y) :=   y + \tilde{z}_{0}(y) . 
		\end{equation}
		Furthermore, one obtains the estimates  
		\begin{equation}\label{Eq1_Theor1_delta Fixed Point} 
			\underset{n \in \mathbb{N} }{\sup}   \left| \left|   \left(  z_{n}(x)   \right)_{n\in \mathbb{N} }  \right|  \right|_{p} \le \delta/2 \quad \text{and} \quad \underset{n \in \mathbb{N} }{\sup}    \left| \left|   \left(  \tilde{z}_{n}(y)   \right)_{n\in \mathbb{N} }  \right|  \right|_{p} \le \delta/2  \quad \text{for all } x,y \in \mathbb{X}_{p},  
		\end{equation}
		which follow in the same way as in the case of inequality \eqref{Eq1_Fixed Point Norm}. 
		Thus, it follows that  
		$$  \left| \left|  h - id  \right|  \right|_{p}   \le   \delta/2 \quad \text{and} \quad  \left| \left|  \tilde{h} - id  \right|  \right|_{p}   \le   \delta/2 .$$

		\paragraph{}
		
		\begin{claim}\label{Claim 1.2}
			The maps $h$ and $\tilde{h}$ are homeomorphisms such that
			\begin{equation}\label{Eq1_Conjugation fg}
				f\circ h   =   h \circ g \quad \text{and} \quad  h^{-1}   =    \tilde{h} .
			\end{equation}

		\end{claim}

		\begin{proof}[Proof of Claim \ref{Claim 1.2}:]
			
			Fix $x, y \in \mathbb{X}_{p}$. 
			For each $n \in \mathbb{N} $,  the uniqueness of the fixed points of the maps $F_{g^{n}(x)}$ and $G_{f^{n}(y)}$ implies that $z_{0} \left(  g^{n}(x)  \right)  = z_{n}(x)$ and $\tilde{z}_{0} \left(  f^{n}(y)  \right)  = \tilde{z}_{n} \left(  y  \right) $. 
			Therefore, one obtains the following relations:   		
			\begin{equation}\label{Eq1_Prop1_Transeference F}
				h \left(  g^{n}(x)  \right)    \underset{\eqref{Eq1_Conjugation Lipschitz Structurally Stable FG}}{=}   g^{n}(x) + z_{n}(x)   \underset{\eqref{Eq1_Theor1_Sh Operator F}}{=}    f^{n} \left(  h(x)  \right)   \quad \text{for all } n \in \mathbb{N} 
			\end{equation}
			and
			\begin{equation}\label{Eq1_Prop1_Transeference G}
				\Tilde{h} \left(  f^{n}(y)  \right)     \underset{\eqref{Eq1_Conjugation Lipschitz Structurally Stable FG}}{=}   f^{n}(y) + \Tilde{z}_{n}(y)   \underset{\eqref{Eq1_Prop1_Sh Operator G}}{=}   g^{n} \circ \Tilde{h} (y)  \quad \text{for all } n \in \mathbb{N}  . 
			\end{equation}
			For $n=1$, these relations yield  
			\begin{equation*}\label{Eq1_Prop1_Conjugations h and tilde h}
				f \circ h   =   h \circ g \quad \text{and} \quad \tilde{h} \circ f   =   g \circ \tilde{h}. 
			\end{equation*}
			This establishes the first equality in equation \eqref{Eq1_Conjugation fg}.

			\paragraph{}
			Next, we show that $h^{-1} = \tilde{h}$. To do this, we need to demonstrate that $\tilde{h} \circ h (x) = x$ and $ h \circ \tilde{h}(y) = y$. 
			
			To prove $\tilde{h}\circ h(x)=x $, we need to show that $\Tilde{z}_{0} \left(  h(x)  \right)  = - z_{0}(x)$ since  
			$$ \Tilde{h}\circ h(x)    =     h(x) + \Tilde{z}_{0} \left(  h(x)  \right)     =    x + z_{0}(x) + \Tilde{z}_{0} \left(  h(x)  \right)      =     x .$$
			In other words, we need to show that the sequence $ \left(  - z_{n}(x)   \right)_{n \in \mathbb{N} }$ is the fixed point of the map $G_{h(x)}$. For each $n \in \mathbb{N} $, one has: 
			\begin{align*}
				\tilde{R}_{j_{n} \left(  h(x)  \right) } \left(  f^{n+1} \left(  h(x)  \right)  - z_{n+1} \left(  x  \right)   \right)    & \underset{\eqref{Eq1_Prop1_Transeference F}}{=}   \tilde{R}_{j_{n} \left(  h(x)  \right) } \left(  g^{n+1} \left(  x  \right)   \right)   \\  
				& =   g^{n}(x)   \underset{\eqref{Eq1_Prop1_Transeference F}}{=}   f^{n} \left(  h (x)  \right)  - z_{n}(x),
			\end{align*}
			which confirms that $ \left(  - z_{n}(x)   \right)_{n \in \mathbb{N} }$ is the fixed point of $G_{h(x)}$. A similar argument shows that $h \circ \tilde{h}(y) = y$.

			\paragraph{}
			Finally, we prove that $h$ and $\tilde{h}$ are continuous or, equivalently, that the maps $z_{0}$ and $\tilde{z}_{0}$ are continuous. As the argument is identical for both maps, we focus on the continuity of $z_{0}$.

			Let $x \in \mathbb{X}_{p}$ be fixed. According to the definition of the map $h$ in equation \eqref{Eq1_Conjugation Lipschitz Structurally Stable FG}, it is sufficient to show that $\underset{x' \to x}{\lim} z_{0} \left(  x'  \right)  = z_{0} \left(  x  \right) $. 
			Recall that $g^{n} \left(  x  \right)  \in R_{i_{n} \left(  x  \right) } \left(  \mathbb{X}_{p}  \right) $ for every $n\in \mathbb{N} $. Since the images $R_{i} \left(  \mathbb{X}_{p}  \right) $ are open for each $i \in \mathcal{I}$, for every $n \in \mathbb{N} $, there exists $  \epsilon_{n}  > 0$ such that $B \left(  g^{m} \left(  x  \right) ,  \epsilon_{n}  \right)  \subseteq R_{i_{m} \left(  x  \right) } \left(  \mathbb{X}_{p}  \right) $ for every $0 \le m \le n$. From the continuity of $g$, there exists $0  < \delta_{n} \le  \epsilon_{n}$ such that $ \left| \left|  g^{m} \left(  x  \right)  - g^{m} \left(  x'  \right)   \right|  \right|_{p} \le  \epsilon_{n}$ for every $0 \le m \le n$, whenever $ \left| \left|  x - x'  \right|  \right|_{p} \le \delta_{n}$. This implies that as $x' \to x$, the quantity 
			$$ n \left(  x, x'  \right)   :=   \sup \left\{  n \in  \mathbb{N} :      g^{m} \left(  x' \right)  \in R_{i_{m}(x)} \left(  \mathbb{X}_{p}  \right)    \text{ for all }  0   \le   m    \le   n  \right\}   $$
			tends to infinity. 
			
			Consequently, given $x' \in \mathbb{X}_{p}$ sufficiently close to $x$, by the definition of the map $z_{0}$, one gets   
			$$ R_{i_{m}(x)} \left(  g^{m+1} \left(  v  \right)  + z_{0}\circ g^{m+1} \left(  v  \right)   \right)    =   g^{m} \left(  v  \right)  + z_{0}\circ g^{m} \left(  v  \right)  , \quad \text{for } v \in  \left\{  x, x' \right\}   $$
			and for every $0 \le m \le n \left(  x, x'  \right) $. 
			By applying the strong triangle inequality,  one deduces that 
			\begin{multline*} 
				\left| \left|  z_{0} \circ g^{m} \left(  x  \right)  - z_{0} \circ g^{m} \left(  x'  \right)   \right|  \right|_{p}    \le   \max  \left\{  \begin{matrix}  \left| \left|  g^{m} \left(  x  \right)  - g^{m}  \left(  x'  \right)   \right|  \right|_{p}, \\   {1\over p} \cdot  \left| \left|  g^{m+1} \left(  x  \right)  - g^{m+1} \left(  x'  \right)   \right|  \right|_{p},  \\  {1\over p}\cdot  \left| \left|  z_{0} \circ g^{m+1} \left(  x  \right)  - z_{0} \circ g^{m+1} \left(  x'  \right)   \right|  \right|_{p}  \end{matrix} \right\}  
			\end{multline*}
			for every $0 \le m \le n \left(  x, x'  \right) $. By iterating this inequality for $0\le m \le k$, with $k \le n \left(  x, x' \right) $, one infers the following inequality: 
			
			\begin{equation}\label{Eq1_Claim1.4_Continuous} 
				\left| \left|  z_{0}(x) - z_{0}(x')  \right|  \right|_{p}   \le   \max  \left\{  \begin{matrix}  \left| \left|  x- x'  \right|  \right|_{p} ,  \\  
					\max \left\{   
					{1\over p^{n} }\cdot  \left| \left|  g^{n}(x) - g^{n}(x')  \right|  \right|_{p} :   1 \le n \le k+1  \right\}   ,   \\ 
					{1\over p^{k+1}}\cdot  \left| \left|  z_{0}\circ g^{k+1}(x) - z_{0}\circ g^{k+1}(x')  \right|  \right|_{p} \end{matrix}   \right\}  .
			\end{equation}

			To minimize the right-hand side of inequality \eqref{Eq1_Claim1.4_Continuous}, fix $ \eta  > 0$ and choose $x'$ sufficiently close to $x$ so that $ \left| \left|  x - x'  \right|  \right|_{p}  <  \eta $. 
			Next, define $ k :=  k \left(  x, x',  \eta  \right)  $ to be the largest integer in $ \left\{  0, \dots, n \left(  x, x' \right)   \right\} $ such that $ \left| \left|  g^{m}(x) - g^{m} \left(  x' \right)   \right|  \right|_{p} \le  \eta $ for all $0 \le m \le k +1$. 
			This results in the following upper bound:
			$$  \left| \left|  z_{0} \left(  x  \right)  - z_{0} \left(  x'  \right)   \right|  \right|_{p}   \le    \max \left\{  \eta ,   {\delta  \over p^{k+1} }  \right\}  .$$
			Since $\underset{x' \to x}{\lim} n \left(  x, x'  \right)  = + \infty $, it follows that $\underset{x' \to x}{\lim} k \left(  x, x',  \eta  \right)  = + \infty $. This conclusion follows from the continuity of $g$, as for every $k \in \mathbb{N} $, there exists $\delta_{k}  >0$ such that $ \left| \left|  g^{m} \left(  x  \right)  - g^{m} \left(  x'  \right)   \right|  \right|_{p} \le  \eta $ for every $0 \le m \le k+1$,
			Finally, letting $ \eta \to 0$, one infers that $\underset{x' \to x}{\lim} z_{0} \left(  x'  \right)  = z_{0} \left(  x  \right) $. Thus, $z_{0}$ is continuous at $x$. 
			Since the choice of $x \in \mathbb{X}_{p}$ was arbitrary, $z_{0}$ is continuous on $\mathbb{X}_{p}$ and, therefore, $h$ is continuous on $\mathbb{X}_{p}$. 
			
			The proof of Claim \ref{Claim 1.2} is complete.
			
		\end{proof}

		\paragraph{Topological Stability:}
		
		It remains to demonstrate that $f$ is topologically stable. To this end, fix $\delta  > 0$ and a map $g: \mathbb{X}_{p} \to \mathbb{X}_{p}$ such that $ \left| \left|  f - g  \right|  \right|_{ \infty } \le \delta$. Define the map $F_{x}: \ell^{ \infty } \left(   \mathbb{N} ,   \mathbb{X}_{p}  \right)  \to \ell^{ \infty } \left(   \mathbb{N} ,   \mathbb{X}_{p}  \right) $ and $h: \mathbb{X}_{p} \to \mathbb{X}_{p}$ as in relations \eqref{Eq1_Theor1_Sh Operator F} and \eqref{Eq1_Conjugation Lipschitz Structurally Stable FG}, respectively. 
		Once again, the maps $F_{x}$ and $h$ are well defined due to the conditions \eqref{Eq1_Inverses Covering Property} and \eqref{Eq1_Inverses Topological Property}. 
		Moreover, since the fixed points of the operators $F_{x}$ and $F_{g(x)}$ are unique, it follows that $f \circ h = h \circ g$. The continuity of $h$ is established in the same way as in the proof of Claim \ref{Claim 1.2}.
		
		Thus, the proof of the theorem is complete. 
		
	\end{proof}

	\paragraph{}
	
	Theorem \ref{Theor1_Dilations in p-adics} implies Corollary \ref{Cor1_Locally Scaling Maps}, upon considering the following representation of $ \left(  p^{-k}, p^{k} \right) $-locally scaling maps.

	\begin{lemma}[\cite{Furno-Natural_Extensions_for_p-adic-beta-shifts_and_other_scaling_maps}, Theorem 1.1]\label{Lem1_Locally Scaling Maps} 
		Let $f: \mathbb{Z}_{p} \to \mathbb{Z}_{p}$ be a $ \left(  p^{-k}, p^{k}  \right) $-locally scaling map with $k \ge 1$. Then, there exists a bijective isometry $w: \mathbb{Z}_{p} \to \mathbb{Z}_{p}$ such that
		\begin{equation}\label{Eq1_Locally Scaling Maps}
			f   =   \sigma^{k} \circ w ,
		\end{equation}
		where $\sigma: \mathbb{Z}_{p} \to \mathbb{Z}_{p}$ is the shift map defined in \eqref{Eq1_Shift Map Zp}.
		
	\end{lemma}

	\begin{proof}[Proof of Corollary \ref{Cor1_Locally Scaling Maps}] 
		Fix $k \ge 1$ and let $f$ be an $ \left(  p^{-k}, p^{k}  \right) $-locally scaling map 
		decomposed as in relation \eqref{Eq1_Locally Scaling Maps}. Note that $\sigma$ admits right inverses given by the contractions $R_{i}: \mathbb{Z}_{p} \to \mathbb{Z}_{p}$ where $ R_{i}(x)= i + px$, with $i \in  \left\{  0, \dots, p-1  \right\} $. 
		Furthermore, the maps $R_{ \mathbf{a}} = w^{-1}\circ R_{a_{1}} \circ \cdots \circ R_{a_{k}}$, where $ \mathbf{a} =  \left(  a_{1}, \dots, a_{k}  \right)  \in  \left\{  0, \dots, p-1  \right\} ^{k}$, are contractions and serve as right inverses of the map $f$. The images $R_{ \mathbf{a}} \left(  \mathbb{Z}_{p}  \right) $, indexed by $ \mathbf{a}$, are open sets and partition $\mathbb{Z}_{p}$. The conclusion follows from the application of Theorem \ref{Theor1_Dilations in p-adics}.
		
		The proof is complete.

	\end{proof}

	\subsection{Proof of Theorem \ref{Theor1_Single Right Inverse does not suffice for shadowing} }

	The purpose of this section is to establish Theorem~\ref{Theor1_Single Right Inverse does not suffice for shadowing}. To this end, Proposition~\ref{Prop1_Topological Stability implies Shadowing} is invoked. A brief outline of its proof is presented first, followed by the proof of Theorem~\ref{Theor1_Single Right Inverse does not suffice for shadowing}.

	\paragraph{} 
	
	The argument used in the proof of Proposition~\ref{Prop1_Topological Stability implies Shadowing} is similar to that in Kawaguchi~\cite{Kawaguchi_Topological_Stability_and_Shadowing_of_Zero-dimensional_Dynamical_Systems}. Therefore, the reader is referred to~\cite{Kawaguchi_Topological_Stability_and_Shadowing_of_Zero-dimensional_Dynamical_Systems} for further details, where the argument can be easily adapted.

	The main idea of the argument is captured in the following two lemmas. To state them, recall that a sequence $ \left(  x_{n}   \right)_{n=0}^{+ \infty }$ (respectively, a finite sequence $ \left(  x_{n}   \right)_{n=0}^{k}$) in $\mathbb{X}_{p}$ is said to be \emph{proper} if its terms are pairwise distinct, that is, if $x_{n}  \neq x_{m}$ for all $n \neq m$ in $\mathbb{N} $ (respectively, for all $n \neq m$ in $ \left\{  0, \dots, k  \right\} $).

	\paragraph{}
	
	The first lemma asserts that for a continuous, nowhere locally constant map $f: \mathbb{X}_{p} \to \mathbb{X}_{p}$,  every $\delta$-pseudo-orbit of $f$ can be approximated by a proper $(3\delta)$-pseudo-orbit.

	\begin{lemma}\label{Lem3_Local Images of Continuous s}
		Let $f: \mathbb{X}_{p} \to \mathbb{X}_{p}$ be a continuous, nowhere locally constant map in $\mathbb{X}_{p}$, and let $\delta  > 0$ be a positive number. Given a $\delta$-pseudo-orbit $ \left(  x_{n}   \right)_{n=0}^{+ \infty }$ of $f$ and $ 0  < \rho   < \delta $, there exists a proper sequence $ \left(  z_{n}   \right)_{n=0}^{+ \infty }$ in $\mathbb{X}_{p}$ satisfying the following properties: 
		\begin{itemize}
			\item For every $n \in \mathbb{N} $ it holds that $ \left| \left|  x_{n} - z_{n}  \right|  \right|_{p}  < \rho$ .
			\item The sequence $ \left(  f \left(  z_{n}  \right)    \right)_{n=0}^{+ \infty }$ is proper, and $ \left| \left|  z_{n+1} - f \left(  z_{n}  \right)   \right|  \right|_{p}  < \max \left\{  \delta, \rho  \right\}  $, for all $n \in \mathbb{N} $.
		\end{itemize} 
		
	\end{lemma}
	
	Lemma \ref{Lem3_Local Images of Continuous s} can be traced back to Walters \cite[Lemma 9]{Walters_On_the_Pseudo_Orbit_Tracing_Property_and_its_relationship_to_Stability} and Kawaguchi \cite[Proof of Lemma 1.1]{Kawaguchi_Topological_Stability_and_Shadowing_of_Zero-dimensional_Dynamical_Systems}, 
	although in those works, the authors focus specifically on the case of homeomorphisms in compact spaces.

	\begin{proof}
		
		Let $f: \mathbb{X}_{p} \to \mathbb{X}_{p}$ be a continuous, nowhere locally constant continuous map. 
		Then, for each $x \in \mathbb{X}_{p}$ and for all $\delta  > 0$, the following holds: 
		\begin{equation}\label{Eq_Local Images of Continous s}    
			\# f \left(  B \left(  x, \delta  \right)   \right)     =   + \infty  .
		\end{equation}

		\paragraph{}
		The sequence $ \left(  z_{i}   \right)_{i=0}^{+ \infty }$ is constructed recursively.
		Begin by selecting $z_{0} \in \mathbb{X}_{p}$ such that $  \left| \left|  x_{0} - z_{0}  \right|  \right|_{p}  < \rho$ and $  \left| \left|  f \left(  x_{0} \right)  - f \left(  z_{0}  \right)   \right|  \right|_{p}  < \rho$.  
		Assume that for some $k \in \mathbb{N} $, we have already chosen the points $z_{0}, z_{1}, \dots, z_{k} \in \mathbb{X}_{p}$ such that the tuples $  \left(  z_{n}   \right)_{n=0}^{k}$ and $ \left(  f \left(  z_{n}  \right)    \right)_{n=0}^{k}$ are proper, and for each $n \in  \left\{  0, \dots, k  \right\} $ it holds that 
		$$  \left| \left|  x_{n} - z_{n}  \right|  \right|_{p}   <   \rho \quad \text{and} \quad   \left| \left|  f \left(  x_{n}  \right)  - f \left(  z_{n}  \right)   \right|  \right|_{p}    <   \rho .$$
		Since $f$ is nowhere locally constant, for any $  \eta  > 0$, the set $f \left(  B \left(  x_{k+1},  \eta  \right)   \right) $ contains infinitely many points. Thus, one can always choose a point $z_{k+1} \in \mathbb{X}_{p}$ such that:
		$$ \left| \left|  x_{k+1} - z_{k+1}  \right|  \right|_{p}  < \rho,    \left| \left|  f \left(  x_{k+1}  \right)  - f \left(  z_{k+1}  \right)   \right|  \right|_{p} < \rho,   \text{and}   f \left(  z_{k+1}  \right)  \not\in  \left\{  f \left(  z_{0}  \right) , \dots, f \left(  z_{k}  \right)   \right\} . $$ 
		By repeating this process, we construct the proper sequence $ \left(  z_{n}   \right)_{n=0}^{+ \infty }$. 
		
		Finally, to ensure that $ \left| \left|  z_{n+1} - f \left(  z_{n}  \right)   \right|  \right|_{p} \le \max \left\{  \delta, \rho  \right\} $ for all $n \in \mathbb{N} $, we apply the triangle inequality:
		\begin{align*}
			\left| \left|  z_{n+1} - f \left(  z_{n}  \right)   \right|  \right|_{p}    &\le    \max \left\{   \left| \left|  z_{n+1} - x_{n+1}  \right|  \right|_{p},    \left| \left|  x_{n+1} - f \left(  x_{n}  \right)   \right|  \right|_{p},    \left| \left|  f \left(  x_{n}  \right)  - f \left(  z_{n}  \right)   \right|  \right|_{p}  \right\}   \\ 
			& <  \max \left\{  \delta, \rho  \right\}  .
		\end{align*}
		From the construction, one has that $ \left| \left|  z_{n+1} - x_{n+1}  \right|  \right|_{p}  < \rho$ and $ \left| \left|  f \left(  x_{n}  \right)  - f \left(  z_{n}  \right)   \right|  \right|_{p} < \rho$, so: $  \left| \left|  z_{n+1} - f \left(  z_{n}  \right)   \right|  \right|_{p} <  \max \left\{  \delta, \rho  \right\}  $. Thus, the sequence $ \left(  z_{n}   \right)_{n=0}^{+ \infty }$ satisfies the required properties. 
		
		The proof of the lemma is complete.

	\end{proof}

	\paragraph{}

	The second lemma asserts that for any two finite sequences of the same length in $\mathbb{X}_{p}$ that are close to each other, there exists a homeomorphism that maps the first sequence to the second. Notably, the closer the two sequences are to each other, the closer the homeomorphism is to the identity map. 
	This property is known as \textit{generalized homogeneity}, and the proof of the lemma follows the standard method for constructing homeomorphisms in contexts like this, particularly in Cantor spaces (see, for instance, \cite[Lemma 1.1]{Kawaguchi_Topological_Stability_and_Shadowing_of_Zero-dimensional_Dynamical_Systems}).

	\begin{lemma}\label{Lem3_Continued Translation of Points}
		Given $ \delta  > 0$ and $k \in \mathbb{N} $, let $ \mathbf{z} =  \left(  z_{n}   \right)_{n=0}^{k}$ and $ \mathbf{y} =  \left(  y_{n}   \right)_{n=0}^{k}$ be two proper sequences in $\mathbb{X}_{p}$ such that $ \left| \left|  z_{n} - y_{n}  \right|  \right|_{p}  < \delta$ for every $0 \le n \le k $. Then, there exists a homeomorphism $\phi: \mathbb{X}_{p} \to \mathbb{X}_{p}$ such that 
		$$ \left| \left|  \phi - id  \right|  \right|_{ \infty }   <    3 \delta \quad \text{and} \quad \phi \left(  y_{n}  \right)    =   z_{n},    \text{ for every } 0   \le   n   \le   k .$$
		
	\end{lemma}

	\paragraph{}
	
	\begin{proof}[Sketch of the Proof (Proposition \ref{Prop1_Topological Stability implies Shadowing})]
		
		Fix $\mathbb{X}_{p} \in  \left\{  \mathbb{Z}_{p}, \mathbb{Q}_{p}  \right\} $, and suppose that $f$ is topologically stable, uniformly continuous, and nowhere locally constant. 
		Given $\delta  > 0$, let $\epsilon =  \epsilon  \left(  \delta  \right)   >0$ be such that for every $g \in C \left(  \mathbb{X}_{p}  \right) $ with $ \left| \left|  f - g  \right|  \right|_{ \infty } \le 3\delta$, there exists a continuous map $h \in C \left(  \mathbb{X}_{p}  \right) $ such that $f \circ h = h \circ g$ and $ \left| \left|  h - id  \right| \right|  \le  \epsilon / 2$. 
		By the uniform continuity of $f$, there exists $ 0  < \rho  < \min \left\{  \delta /3 ,  \epsilon / 3  \right\} $ such that if $ \left| \left|  u -v  \right|  \right|_{p}  < \rho$, then $ \left| \left|  f(u) - f(v)  \right|  \right|_{p} \:< \delta /3 $ for every $u, v \in \mathbb{X}_{p}$.

		To prove that $f$ $ \epsilon $-shadows a  $ \left(  \delta / 3  \right) $-pseudo-orbit $ \left(  x_{n}   \right)_{n=0}^{+ \infty }$, it suffices to show that $f$ $ \epsilon $-shadows the proper sequence $  \left(  z_{n}   \right)_{n=0}^{+ \infty }$ obtained by Lemma \ref{Lem3_Local Images of Continuous s}, which is $\rho$-close to $ \left(  x_{n}   \right)_{n =0}^{+ \infty }$. 
		By the compactness of the ball $B \left(  z_{0},  \epsilon  \right) $, a standard compactness argument guarantees that it suffices to prove that $f$ $ \epsilon $-shadows every finite subsequence $ \left(  z_{n}   \right)_{n=0}^{k}$, for $k \in \mathbb{N} $. 
		
		To this end, Lemma \ref{Lem3_Continued Translation of Points} yields a homeomorphism $\phi_{k}: \mathbb{X}_{p} \to \mathbb{X}_{p}$ such that $  \left| \left|  \phi_{k} - id  \right|  \right|_{ \infty }  <  3 \delta$ and $\phi_{k} \left(  f \left(  z_{n}  \right)   \right)  = z_{n+1}$ for all $0 \le n \le k + 1$.
		Then, define the continuous map $g_{k} := \phi_{k} \circ f $, which satisfies $g_{k}^{n} \left(  z_{0}  \right)  = z_{n}$, for all $0 \le n \le k+1$, and $ \left| \left|  g_{k} - f  \right|  \right|_{ \infty } \le 3 \delta$. 
		Since $f$ is topologically stable, there exists a homeomorphism $h_{k}: \mathbb{X}_{p} \to \mathbb{X}_{p}$ such that $h_{k} \circ g_{k} = f \circ h_{k}$. 
		The proof concludes upon verifying that the finite orbit $  \left(  f^{n} \left(  h_{k} \left(  z_{0}  \right)   \right)    \right)_{n =0}^{k}$ $ \epsilon $-shadows the finite $(\delta/3)$-pseudo-orbit $ \left(  z_{n}   \right)_{n=0}^{k}$. 
		
	\end{proof}

	\paragraph{}

	\begin{proof}[Proof of Theorem \ref{Theor1_Single Right Inverse does not suffice for shadowing}]
		The proof is carried out in two steps. First, an example of a map
		$f \colon \mathbb{Z}_{p} \to \mathbb{Z}_{p}$ satisfying the conclusions of the theorem is constructed.
		Next, this example is extended to the case $\mathbb{X}_{p} = \mathbb{Q}_{p}$ by suitably
		expanding the map $f$ from $\mathbb{Z}_{p}$ to $\mathbb{Q}_{p}$.

		\paragraph{}
		
		Fix a prime number $p \ge 2$. Without loss of generality, assume that $p \ge 3$; the argument for the case $p=2$ can be adapted in the same way. 
		In what follows, the space of infinite words $ \left\{  0, \dots, p-1  \right\} ^{\mathbb{N} }$ is identified with $\mathbb{Z}_{p}$ via the map $P:  \left\{  0, \dots, p-1  \right\} ^{\mathbb{N} } \to \mathbb{Z}_{p}$, $P \left(   \left(  a_{i}   \right)_{i=0}^{+ \infty }  \right)  = \sum_{i=0}^{+ \infty } a_{i} p^{i}$. The space $ \left\{  0, \dots, p-1  \right\} ^{\mathbb{N} }$ is endowed with the metric 
		$$d_{p} \left(   \left(  a_{i}   \right)_{i=0}^{+ \infty },  \left(  b_{i}   \right)_{i=0}^{+ \infty }  \right)    =     \left| \left|  P \left(   \left(  a_{i}   \right)_{i=0}^{+ \infty }  \right)  - P \left(   \left(  b_{i}   \right)_{i=0}^{+ \infty }  \right)   \right|  \right|_{p} . $$
		Note that the topology of $ \left\{  0, \dots, p-1  \right\} ^{\mathbb{N} }$ induced by $d_{p} \left(  \cdot, \cdot  \right) $ on $ \left\{  0, 1, \dots, p-1  \right\} ^{\mathbb{N} }$ coincides with the product topology, where the alphabet $ \left\{  0, \dots, p-1  \right\} $ is equipped with the discrete topology.

		It is known that there exists a nonempty class of continuous, non-shadowing maps in $ \left\{  0, \dots, p-1  \right\} ^{\mathbb{N} }$. 
		Recall that a closed subset $ X \subseteq  \left\{  0, \dots, p-1  \right\} ^{\mathbb{N} }$ is called a subshift if it is invariant under the shift map $ Z:  \left\{  0, \dots, p-1  \right\} ^{\mathbb{N} } \to  \left\{  0, \dots, p-1  \right\} ^{ \mathbb{N} }$, i.e., $Z \left(  X  \right)  \subseteq X$. 
		It is known that the restriction of the shift map $Z: X \to X$ on a subshift $X$ is shadowing if and only if the system $ \left(  X, Z  \right) $ is of finite type (see \cite[Chapter 15, Section 4]{Aoki-Topological_Dynamics} and \cite[Theorem 1]{Walters_On_the_Pseudo_Orbit_Tracing_Property_and_its_relationship_to_Stability}). 
		Let us now fix a subshift $ \left(  X, Z  \right) $ that is not of finite type, and hence not shadowing. Since $X$ is a Cantor space, 
		there exists a homeomorphism $w: X \to \mathbb{Z}_{p}$. 
		Define the map $\zeta: \mathbb{Z}_{p} \to \mathbb{Z}_{p}$ by 
		\begin{equation}\label{Eq_Theor1_SRI Non Shadowing Map}   
			\zeta  = w \circ Z \circ w^{-1}.
		\end{equation} 
		Since $Z$ is not shadowing, it follows that $\zeta$ is also not shadowing. 
		
		Next, define the map $f: \mathbb{Z}_{p} \to \mathbb{Z}_{p}$ as follows. For $x \in \mathbb{Z}_{p}$, decompose $x$ as $x = a_{0} \left(  x  \right)  + a_{1} \left(  x  \right)  p + z p^{2}$, where $ a_{0} \left(  x  \right) , a_{1} \left(  x  \right)  \in  \left\{  0, \dots, p-1  \right\} $ and $z =  \sum_{i=0}^{+ \infty } a_{i+2} \left(  x  \right)  p^{i}$. Then set 
		$$ f(x)  := \quad \begin{cases}
			x , & \text{if } a_{0} \left(  x  \right)  = 0   \\ 
			z, & \text{if } a_{0} \left(  x  \right)  \neq 0 \text{ and } a_{1} \left(  x  \right)  = 0   \\ 
			a_{0} \left(  x  \right)  +  \left(  a_{1} \left(  x  \right)  + 1  \right)  p +  \zeta(z) p^{2} , & \text{if } a_{0} \left(  x  \right)   \neq 0 \text{ and } 1\le a_{1} \left(  x  \right)  \le p-2   \\ 
			a_{0} \left(  x  \right)  + p +  \zeta(z) p^{2}, & \text{if } a_{0} \left(  x  \right)  \neq 0 \text{ and } a_{1} \left(  x  \right)  = p-1 ,
		\end{cases}      
		$$
		where the map $\zeta$ is defined in \eqref{Eq_Theor1_SRI Non Shadowing Map}. 
		Since the decomposition $x = a_{0} \left(  x  \right)  + a_{1} \left(  x  \right)  p +  zp^{2} $ is unique, the map $f$ is well-defined. Notice that the map $f$ is continuous and admits the contractions $R_{a}(x) = a +  p^{2} x$ as right inverses for all $a \in  \left\{   1, \dots, p-1  \right\} $. 
		To show that $f$ is not shadowing, it suffices to observe that any $\delta$-pseudo-orbit $ \left(  x_{n}   \right)_{n \in  \mathbb{N} }$ of $\zeta$ is $ \epsilon $-shadowed by $\zeta$ if and only if the $ \left(  \delta p^{-2}  \right) $-pseudo-orbit $ \left\{  a + b_{n}p + x_{n} p^{2}  \right\}_{n=0}^{+ \infty }$ of $f$, where $b_{n} = u +1$ when $n = m \cdot (p-1) + u$ for $u \in  \left\{  0,1, \dots, p-2  \right\} $, is $ \left(   \epsilon p^{-2}  \right) $-shadowed by $f$ for any choice of $a \neq 0$. 
		Since $\zeta$ is not shadowing, this implies that $f$ is also not shadowing. 
		
		Now, we prove that $f$ is not topologically stable. Since $f$ is not shadowing, by Proposition \ref{Prop1_Topological Stability implies Shadowing}, it suffices to show that $f$ is nowhere locally constant, i.e., that $f$ has no locally constant points. 
		It is clear that if $ a_{0} \left(  x  \right)  = 0$ or $a_{1} \left(  x  \right) = 0$, then $a_{0} \left(  x  \right)   + a_{1} \left(  x  \right)  p + z p^{2}$ is not a locally constant point of $f$. 
		For the remaining cases, fix $x= a_{0} \left(  x  \right)  +  a_{1} \left(  x  \right)  p + z p^{2}$ with $a_{0} \left(  x  \right)  \neq 0$ and $ a_{1} \left(  x  \right)  \neq 0$. Then, for every $y$ sufficiently close to $x$, with $x \neq y$, it holds that $0  <  d_{p} \left(  w^{-1} \left(  x  \right) ,  w^{-1} \left(  y  \right)   \right)  \le p^{-2}$. 
		Since $w: X \to \mathbb{Z}_{p}$ is a homeomorphism and $Z$ is the shift map, it follows that $Z\circ w^{-1}(x) \neq Z \circ w^{-1}(y)$, which implies that $\zeta \left(  x  \right)  \neq \zeta \left(  y  \right) $. Thus $x$ is not a locally constant point of $f$.  
		This establishes that $f$ is nowhere locally constant. 
		
		Finally, we prove that $f$ is not Lipschitz structurally stable. Fix $\delta  < 1/p$ and a $\delta$-Lipschitz map $\phi \in Lip_{\delta} \left(  \mathbb{Z}_{p}  \right) $ such that $\phi \left(  x  \right)  \neq 0$ when $  \left| \left|  x  \right|  \right|_{p} \le p^{-1}$ and $\phi(x) = 0 $ otherwise. 
		Define the map $g = f + \phi$. We show that $f$ is not topologically conjugate to $g$. 
		Assume, for contradiction, that a homeomorphism $h: \mathbb{Z}_{p} \to \mathbb{Z}_{p}$ exists such that $f \circ h = h \circ g$. 
		Then, for every $x \in \mathbb{Z}_{p}$ with $ \left| \left|  h \left(  x  \right)   \right|  \right|_{p} \le p^{-1}$, one has $h(x) = h \left(  f(x) + \phi \left(  x  \right)   \right) $. Since $h$ is a homeomorphism, there exists an open set $V \subseteq \mathbb{Z}_{p}$ such that $x  = f(x) + \phi(x)$, for all $x \in V$. This conclusion, however, contradicts the construction of $f$ and the choice of the $\delta$-Lipschitz map $\phi$. Consequently, the map $f$ fails to be Lipschitz structurally stable.

		\paragraph{} 
		
		In the case $\mathbb{X}_{p} = \mathbb{Q}_{p}$, define the map $\tilde{f}: \mathbb{Q}_{p} \to \mathbb{Q}_{p}$ by 
		$$ \tilde{f} \left(  x  \right)  :=   \begin{cases}
			f \left(  x  \right) , & \text{if } u_{p} \left(  x  \right)  \ge 0 ,  \\  
			p^{u_{p} \left(  x  \right) } f \left(  p^{-u_{p} \left(  x  \right) } x  \right) ,  & \text{if } u_{p} \left(  x  \right)  \le - 1 .  
		\end{cases}  $$
		Observe that $\tilde{f}$ admits, as right inverses, the contractions 
		$$ \tilde{R}_{a} \left(  x  \right)    =    \begin{cases} 
			R_{a} \left(  x  \right)  , & \text{if } u_{p} \left(  x  \right)  \ge 0,  \\  
			a  p^{u_{p} \left(  x  \right)  } + p^{2} x, & \text{if } u_{p} \left(  x  \right)  \le - 1, 
		\end{cases}
		$$
		for all $a \in  \left\{  1, \dots, p-1  \right\} $. Indeed, if $ x \in \mathbb{Z}_{p}$, the claim follows immediately from the definition. If $x \in \mathbb{Q}_{p} \backslash \mathbb{Z}_{p}$, a straightforward computation gives 
		\begin{align*}
			\tilde{f} \circ \tilde{R}_{a}  \left(  x  \right)    =    p^{u_{p} \left(  x  \right) } f \left(  a + p^{2 - u_{p} \left(  x  \right) } x  \right)    =   p^{u_{p} \left(  x  \right) } f \circ R_{a} \left(  p^{- u_{p} \left(  x \right) } x  \right)     =    x. 
		\end{align*}
		Arguing analogously to the case $\mathbb{X}_{p} = \mathbb{Z}_{p}$, it follows that the map $\tilde{f}$
		is neither shadowing, nor topologically stable, nor Lipschitz structurally stable.
		
		The proof is complete.
	\end{proof}

	\subsection{Proof of Theorem \ref{Theor3_Contractions in p-adics}}\label{Sec_Proof of Theorem Left Inverse}
	
	Throughout this section, the image $R \left(  \mathbb{X}_{p}  \right) $ is said to be \textit{$\rho$-open} for some $\rho  > 0$ if, for every $x \in R \left(  \mathbb{X}_{p}  \right) $, the ball $B \left(  x, \rho  \right) $ is contained in $R \left(  \mathbb{X}_{p}  \right) $, that is, $B \left(  x, \rho  \right)  \subseteq R \left(  \mathbb{X}_{p}  \right) $.

	\paragraph{}
	
	The proof of Theorem~\ref{Theor3_Contractions in p-adics} also relies on the following three lemmas. The first shows that a bi-Lipschitz contraction with open image is an open mapping; moreover, if $ R $ is scaling, then its image is $\rho$-open.

	\begin{lemma}\label{Lem2_Open bi-Lipschitz Contractions}
		Let $R: \mathbb{X}_{p} \to \mathbb{X}_{p}$, with $\mathbb{X}_{p} \in \{\mathbb{Z}_{p}, \mathbb{Q}_{p}\}$, be a contraction satisfying the bi-Lipschitz inequality \eqref{Eq2_Bi-Lipschitz Contractions}. 
		If the image $R(\mathbb{X}_{p})$ is open, then $R$ is an open mapping. 
		In particular, if $R$ is scaling, then its image $R(\mathbb{X}_{p})$ is $\rho$-open for some $\rho > 0$.
		
	\end{lemma}
	
	\begin{proof}
		Let $R : \mathbb{X}_{p} \to \mathbb{X}_{p}$ be a contraction satisfying the bi-Lipschitz inequality \eqref{Eq2_Bi-Lipschitz Contractions} with constants $0  < c_{1} \le c_{2}  < 1$. 
		
		\paragraph{} 
		
		Assume that the image $R \left(  \mathbb{X}_{p}  \right) $ is open. To prove that $R$ is an open map, it suffices to show that for every ball $B \left(  x, r  \right)  \subseteq \mathbb{X}_{p}$, the image $R \left(  B \left(  x, r  \right)   \right) $ is open.

		Fix $x \in \mathbb{X}_{p}$ and $r  > 0$.  Let $y_{0} \in R \left(  B \left(  x, r  \right)   \right) $. 
		It will be shown that there exists $ \eta  > 0$ such that $B \left(  y_{0},  \eta  \right)  \subseteq R \left(  B \left(  x, r  \right)   \right) $. 
		Choose $x_{0} \in B \left(  x, r  \right) $ such that $R \left(  x_{0}  \right)  = y_{0}$. Since $R \left(  \mathbb{X}_{p}  \right) $ is open, there exists $ \epsilon  > 0$ such that $B \left(  y_{0},  \epsilon  \right)  \subseteq R \left(  \mathbb{X}_{p}  \right) $. 
		Choose $ \eta  < \min \left\{  c_{1} r,    \epsilon  \right\} $ and fix $y_{1} \in B \left(  y_{0},  \eta  \right) $. Since $y_{1} \in B \left(  y_{0},  \epsilon  \right)  \subseteq R \left(  \mathbb{X}_{p}  \right) $, there exists $x_{1} \in \mathbb{X}_{p}$ such that $R \left(  x_{1}  \right)  = y_{1}$. Since $ \eta < c_{1} r$, the bi-Lipschitz inequality~\eqref{Eq2_Bi-Lipschitz Contractions} implies that $ \left| \left|  x_{0} - x_{1}  \right|  \right|_{p} < r$. Consequently, $y_{1} \in R \left(  B \left(  x, r  \right)   \right) $. 
		
		As the choice of $y_{1} \in B \left(  y_{0},  \eta  \right) $ was arbitrary, it follows that $B \left(  y_{0},  \eta  \right)  \subseteq R \left(  B \left(  x, r  \right)   \right) $. Hence $R \left(  B \left(  x, r  \right)   \right) $ is open, and thus $R$ is an open map.

		\paragraph{Scaling case.}
		Assume, in addition, that $R$ is scaling, and let $c_{1} = p^{-l}$ for some integer $l > 1$. It suffices to prove 
		that $R$ is $ \left(  p^{-n_{0}}, p^{-l}  \right) $-locally scaling, that is, there exists an integer $n_{0} \ge 0$ such that, for all $x_{1},x_{2} \in \mathbb{X}_{p}$ satisfying  $  \left| \left|  x_{1}-x_{2}  \right|  \right|_{p}  \le  p^{-n_{0}}$, one has $  \left| \left|  R  \left(  x_{1}  \right)  - R \left(  x_{2}  \right)  \right|  \right|_{p} = p^{-l}  \left| \left|  x_{1} - x_{2}  \right|  \right|_{p}$.  
		Under this condition, the map $p^{l}R$ is an isometry on every ball of the form $B \left(  x,p^{-n_{0}}  \right) $, and hence $R \left(  \mathbb{X}_{p}  \right) $ is $\rho$-open with $\rho = p^{-l-n_{0}}$.

		\paragraph{} 
		
		It is first shown that the sequence $\kappa \left(  p^{-n}  \right) $ is decreasing in $n$. Indeed, one has
		\begin{align*}
			\kappa \left(  p^{- \left(  n+1  \right) }  \right)    &=    \left| \left|  R \left(  p^{n+1}  \right)  - R \left(  0  \right)   \right|  \right|_{p}  \\ 
			&\le   \max \left\{   \left| \left|  R \left(  p^{n+1}  \right)  - R \left(  p^{n}  \right)   \right|  \right|_{p},  \left| \left|  R \left(  p^{n}  \right)  - R \left(  0  \right)   \right|  \right|_{p}  \right\}     =   \kappa \left(  p^{-n}  \right) , 
		\end{align*}
		where the last equality follows from the assumption that $R$ is scaling.

		Next, it is shown that $\kappa \left(  p^{-n}  \right) $ is strictly decreasing. Suppose that $\kappa \left(  p^{-n}  \right)  = \kappa \left(  p^{-n -1}  \right) $ for some $n$, and consider the set  
		$$ A   =     \left\{  a p^{n} + b p^{n+1} :   a = 0, 1, \dots, p-1, \text{ and } b = 0, 1, \dots, p-1  \right\}  . $$ 
		Then $ \left| \left|  R \left(  x  \right)  - R \left(  y  \right)   \right|  \right|_{p} = \kappa \left(  p^{-n}  \right) $ for all distinct $x,y \in A $, so that the set $B =  \left\{  R \left(  x  \right)  :   x \in A  \right\} $ consists of $p^{2}$ points that are pairwise equidistant. This contradicts the fact that such a subset of $\mathbb{X}_{p}$ has cardinality at most $p$. Hence, $\kappa \left(  p^{-n}  \right) $ is strictly decreasing. 
		
		It follows that $\kappa \left(  p^{-n}  \right) / p^{-n}$ is decreasing. Indeed, since $\kappa \left(  p^{-n}  \right) $ is strictly decreasing, one has 
		$$ {\kappa \left(  p^{-n -1}  \right)  \over p^{-n -1} }   \le   { \kappa \left(  p^{-n}  \right) / p \over p^{-n - 1} }   \le   { \kappa \left(  p^{-n}  \right)  \over p^{-n}} . $$
		Since $R$ is bi-Lipschitz, the sequence $\kappa \left(  p^{-n}  \right)  / p^{-n}$ takes values in $ \left\{  p^{k} :   p \in \mathbb{Z}  \right\} $ and is bounded below by $c_{1} > 0$. Therefore, it is eventually constant, which yields the desired property. 
		
		Consequently, $R(\mathbb{X}_{p})$ is $\rho$-open for some $\rho>0$.

	\end{proof}

	The following two lemmas show that the analytic properties of a bi-Lipschitz contraction remain unaffected under the addition of sufficiently small Lipschitz perturbations. More precisely, the address the effect of adding a $\delta$-Lipschitz map, for $\delta > 0$ small enough. The first lemma concerns the case of a bi-Lipschitz contraction with open image. 
	
	\begin{lemma}\label{Lem2_Open bi-Lipschitz Contraction plus Lipschitz}
		Let $R: \mathbb{X}_{p} \to \mathbb{X}_{p}$, with $\mathbb{X}_{p} \in  \left\{  \mathbb{Z}_{p}, \mathbb{Q}_{p}  \right\} $, be a contraction satisfying the bi-Lipschitz inequality \eqref{Eq2_Bi-Lipschitz Contractions} with constants $0  < c_{1} \le c_{2}  < 1$. 
		Then, for every $0  < \delta \le  c_{1}/2  $ and every $\phi \in Lip_{\delta} \left(  \mathbb{X}_{p}  \right) $, the map $T = R + \phi$ is a bi-Lipschitz contraction satisfying the bi-Lipschitz inequality \eqref{Eq2_Bi-Lipschitz Contractions} with the same constants $c_{1}$ and $c_{2}$. 
		Moreover, if the image $R \left(  \mathbb{X}_{p}  \right) $ is a $\rho$-open set for some $\rho  > 0$ and $0  < \delta \le \min \left\{  c_{1}/2,   \rho/2  \right\} $, then $ T \left(  \mathbb{X}_{p}  \right)  = R \left(  \mathbb{X}_{p}  \right) $.
	\end{lemma}

	\begin{proof}
		
		Fix $0  < \delta  <  c_{1}/2$ and $\phi \in Lip_{\delta} \left(  \mathbb{X}_{p}  \right) $. A direct computation shows that the map $T:= R + \phi$ is a contraction and satisfies the bi-Lipschitz inequality \eqref{Eq2_Bi-Lipschitz Contractions} with the same constants $0  < c_{1} \le c_{2}  < 1$.

		\paragraph{}
		Assume that the image $R \left(  \mathbb{X}_{p}  \right) $ is a $\rho$-open set for some $\rho  > 0$, and that $ 0 < \delta  < \min \left\{  c_{1}/2, \rho / 2  \right\} $. 
		
		It will be shown that $T \left(  \mathbb{X}_{p}  \right)  = R \left(  \mathbb{X}_{p}  \right) $. The inclusion $T \left(  \mathbb{X}_{p}  \right)  \subseteq R \left(  \mathbb{X}_{p}  \right) $ follows immediately. Indeed, for every $x \in \mathbb{X}_{p}$, $T \left(  x  \right)  = R \left(  x  \right)  + \phi \left(  x  \right)  \in B \left(  R \left(  x  \right) , \rho  \right)  \subseteq R \left(  \mathbb{X}_{p}  \right) $, since $ \left| \left|  \phi \left(  x \right)   \right|  \right|_{p} \le \delta < \rho$.   
		To prove the inverse inclusion, fix $z \in R \left(  \mathbb{X}_{p}  \right) $. By the bi-Lipschitz inequality \eqref{Eq2_Bi-Lipschitz Contractions}, the map $R$ admits a continuous left inverse $L: R \left(  \mathbb{X}_{p}  \right)  \to \mathbb{X}_{p}$.  
		Define the map $H_{z}: R \left(  \mathbb{X}_{p}  \right)  \to R \left(  \mathbb{X}_{p}  \right) $ by  
		$$ H_{z} \left(  y  \right)    =   z - \phi \circ L (y) .$$ 
		The choice $\delta < \rho$ ensures that $H_{z} \left(  R \left(  \mathbb{X}_{p}  \right)   \right)  \subseteq R \left(  \mathbb{X}_{p}  \right) $. 
		Moreover, the bi-Lipschitz inequality together with the bound on $\delta$ implies that $H_{z}$ is a contraction. Since $R \left(  \mathbb{X}_{p}  \right) $ is complete, the contraction mapping theorem guarantees the existence of a fixed point $y_{0} \in R \left(  \mathbb{X}_{p}  \right) $ such that $H_{z} \left(  y_{0}  \right)  = y_{0}$. Let $x_{0}  \in \mathbb{X}_{p}$ satisfy $R \left(  x_{0}  \right)  = y_{0}$. Then $ z - \phi \left(  x_{0}  \right)  = R \left(  x_{0}  \right) $, and hence $T \left(  x_{0}  \right)   = R \left(  x_{0}  \right)   + \phi \left(  x_{0}  \right)  = z$. 
		Since $z \in R \left(  \mathbb{X}_{p}  \right) $ was arbitrary, it follows that $R \left(  \mathbb{X}_{p}  \right)  \subseteq T \left(  \mathbb{X}_{p}  \right) $. 
		
		Combining the two inclusions yields $T \left(  \mathbb{X}_{p}  \right)  = R \left(  \mathbb{X}_{p}  \right) $, which completes the proof. 
		
	\end{proof}

	The final lemma addresses perturbations of scaling bi-Lipschitz contractions.
	
	\begin{lemma}\label{Lem2_Scaling Contraction plus Lipschitz}
		Let $R: \mathbb{X}_{p} \to \mathbb{X}_{p}$, with $\mathbb{X}_{p} \in  \left\{  \mathbb{Z}_{p}, \mathbb{Q}_{p}  \right\} $, be a contraction satisfying the left-hand side of the bi-Lipschitz inequality \eqref{Eq2_Bi-Lipschitz Contractions} with constant $0  < c_{1}   < 1$. 
		If $R$ is a scaling contraction, then for every $0  < \delta \le  c_{1}/2  $ and for every $\phi \in Lip_{\delta} \left(  \mathbb{X}_{p}  \right) $, the map $T = R + \phi$ is scaling and satisfies 
		\begin{equation}\label{EqLem_Scaling Maps Induction}
			\left| \left|  T^{n}(x) - T^{n}(y)  \right|  \right|_{p}   =    \left| \left|  R^{n}(x) - R^{n}(y)  \right|  \right|_{p}, \quad \text{for every } x,y \in \mathbb{X}_{p} \text{ and } n \in \mathbb{N} . 
		\end{equation} 
		
	\end{lemma}

	\begin{proof}
		
		Assume that $R$ is a scaling contraction satisfying the bi-Lipschitz inequality \eqref{Eq2_Bi-Lipschitz Contractions}. Fix $0  < \delta  <  c_{1}/2$ and $\phi \in Lip_{\delta} \left(  \mathbb{X}_{p}  \right) $. The fact that the map $T = R + \phi$ is scaling follows from identity \eqref{EqLem_Scaling Maps Induction}.

		\paragraph{}
		It therefore remains to establish relation \eqref{EqLem_Scaling Maps Induction}. To this end, it suffices to verify that, for all $x_{1}, x_{2}, y_{1}, y_{2} \in \mathbb{X}_{p}$ satisfying $ \left| \left|  y_{1} - y_{2}  \right|  \right|_{p} =  \left| \left|  x_{1} - x_{2}  \right|  \right|_{p}$, the equality $  \left| \left|  T \left(  y_{1}  \right)  - T \left(  y_{2}  \right)   \right|  \right|_{p} =  \left| \left|  R \left(  x_{1}  \right)  - R \left(  x_{2}  \right)   \right|  \right|_{p} $ holds. 
		Once this property is established, the desired conclusion follows by a straightforward induction. 
		
		Fix $x_{1}, x_{2}, y_{1}, y_{2}\in \mathbb{X}_{p}$ satisfying $ \left| \left|  x_{1} - x_{2}  \right|  \right|_{p} =  \left| \left|  y_{1} - y_{2}  \right|  \right|_{p}$.  Then, since $\phi$ is a $\delta$-Lipschitz map with $\delta < c_{1}$ and by applying the strong triangle inequality,  
		$$ \left| \left|  T \left(  y_{1}  \right)  - T \left(  y_{2}  \right)   \right|  \right|_{p}    =    \left| \left|  R \left(   y_{1}  \right)  - R \left(  y_{2}  \right)  + \phi \left(  y_{1}  \right)  - \phi \left(  y_{2}  \right)   \right|  \right|_{p}   =    \left| \left|  R \left(  y_{1}  \right)  - R \left(  y_{2}  \right)   \right|  \right|_{p} . $$

		Because $R$ is scaling, there exists a map $\kappa:  \left\{  p^{n}:   n \in \mathbb{Z}  \right\}  \to  \left\{  p^{n} :   n \in \mathbb{Z}  \right\}  \cup  \left\{  0  \right\}  $ satisfying relation~\eqref{Eq_Scaling Maps}. Consequently,    
		$$  \left| \left|  R \left(  y_{1}  \right)  - R \left(  y_{2}  \right)   \right|  \right|_{p}   \underset{\eqref{Eq_Scaling Maps}}{=}   \kappa \left(   \left| \left|  y_{1} - y_{2}  \right|  \right|_{p}  \right)    =   \kappa \left(   \left| \left|  x_{1} - x_{2}  \right|  \right|_{p}  \right)    \underset{\eqref{Eq_Scaling Maps}}{=}    \left| \left|  R \left(  x_{1}  \right)  - R \left(  x_{2}  \right)   \right|  \right|_{p}. $$
		This completes the proof.

	\end{proof}

	\paragraph{}

	\begin{proof}[Proof of Theorem \ref{Theor3_Contractions in p-adics}] 
		
		Fix $\mathbb{X}_{p} \in  \left\{  \mathbb{Z}_{p}, \mathbb{Q}_{p}  \right\} $ and let $R: \mathbb{X}_{p} \to \mathbb{X}_{p}$ be a bi-Lipschitz contraction on $\mathbb{X}_{p}$ that satisfies the bi-Lipschitz inequality \eqref{Eq2_Bi-Lipschitz Contractions} with positive constants $0  < c_{1} \le c_{2}  < 1$. Let $x_{R} \in \mathbb{X}_{p}$ denote the unique fixed point of $R$.

		\paragraph{}
		
		Assume, without loss of generality, that $R \left(  \mathbb{X}_{p}  \right) $ is $\rho$-open for some $\rho > 0$. 
		This assumption is satisfied in the present setting. 
		In the case of Point~(1) of Theorem~\ref{Theor3_Contractions in p-adics}, the condition holds because $R \left(  \mathbb{Z}_{p}  \right) $ is open by hypothesis and compact by continuity. 
		For Point~(2), where $R$ is assumed to be scaling, Lemma~\ref{Lem2_Open bi-Lipschitz Contractions} ensures that the image $R \left(  \mathbb{X}_{p}  \right) $ is $\rho$-open.

		\paragraph{}
		
		We now present a unified argument covering both Point~(1) and Point~(2). The specific assumptions required to complete the proof in each case are invoked only at the final stage, in the respective concluding paragraphs.

		\paragraph{}

		Fix $0  < \delta  <  \min  \left\{  \rho/2, c_{1}/2  \right\} $, $\phi \in Lip_{\delta} \left(  \mathbb{X}_{p}  \right) $, and set $T = R + \phi$. By Lemma \ref{Lem2_Open bi-Lipschitz Contraction plus Lipschitz}, the map $T$ is also a bi-Lipschitz contraction satisfying   
		\begin{equation}\label{Eq3_Contraction T Lower Bound}
			c_{1} \cdot  \left| \left|  x - y  \right|  \right|_{p}   \le    \left| \left|  T(x) - T(y)  \right|  \right|_{p}    \le   c_{2} \cdot  \left| \left|  x - y  \right|  \right|_{p}, \quad \text{for all } x,y \in \mathbb{X}_{p}, 
		\end{equation}
		where the constants $c_{1}$ and $ c_{2} $ are the same as in the inequality \eqref{Eq2_Bi-Lipschitz Contractions}. Denote by $x_{T} \in \mathbb{X}_{p}$ the unique fixed point of $T$. 
		
		Lemmas \ref{Lem2_Open bi-Lipschitz Contractions} and \ref{Lem2_Open bi-Lipschitz Contraction plus Lipschitz} together imply that both $R$ and $T$ are open maps with  $ T \left(  \mathbb{X}_{p}  \right)  = R \left(  \mathbb{X}_{p}  \right)  $. Thus, one can define the set  
		\begin{equation}\label{Eq2_Domain of the Contraction}  
			\mathcal{V} :=    \mathbb{X}_{p} \backslash R \left(  \mathbb{X}_{p}  \right)    =   \mathbb{X}_{p} \backslash T \left(  \mathbb{X}_{p}  \right)  .
		\end{equation}
		Note that $\mathcal{V}$ is a $\rho$-open set, as it is the complement of the $\rho$-open set $R \left(  \mathbb{X}_{p}  \right) $ in $\mathbb{X}_{p}$. 
		We also introduce the sets 
		\begin{equation}\label{Eq2_Contraction Domain B}
			A_{R} :=   \bigcup_{n=0}^{+ \infty } R^{n} \left(  \mathcal{V}  \right) , \quad B_{R} :=   \mathbb{X}_{p} \backslash A_{R}, \quad A_{T} :=   \bigcup_{n=0}^{+ \infty } T^{n} \left(  \mathcal{V}  \right) ,   \text{and }   B_{T} :=   \mathbb{X}_{p} \backslash A_{T}.
		\end{equation}

		\paragraph{}
		
		The proof will now proceed by constructing a homeomorphism $h: \mathbb{X}_{p} \to \mathbb{X}_{p}$ such that $R \circ h = h \circ T$.  Specifically, $h$ is given by the formula 
		\begin{equation}\label{Eq3_Conjugation Fundamental Domains A and B}
			h \left(  x  \right)  :=   \begin{cases}
				h_{1} \left(  x  \right) ,  & \text{if } x \in A_{T},  \\ 
				h_{2} \left(  x  \right) ,   & \text{if } x \in B_{T},  \\  
			\end{cases}
		\end{equation}
		where $h_{1}: A_{T} \to A_{R}$ and $h_{2}: B_{T} \to B_{R}$ are bijections, which will be specified below. Along with the definitions of $h_{1}$ and $h_{2}$, we will show that 
		\begin{equation}\label{Eq3_Conjugation h delta close to Identity}
			\left| \left|  h_{1}(x) - x  \right|  \right|_{p}   \le   \delta,   \text{ for all } x \in A_{T}, \text{ and }    \left| \left|  h_{2}(x) - x  \right|  \right|_{p}   \le   \delta,   \text{ for all } x \in B_{T}. 
		\end{equation}
		Finally, we will prove the continuity of $h$ and $h^{-1}$.

		To define the maps $h_{1}$ and $h_{2}$, we use the following lemma, which can be readily proved.   
		
		\begin{lemma}\label{Lem2_Domain of the Contraction}
			Let $R: \mathbb{X}_{p} \to \mathbb{X}_{p}$ be a left invertible contraction. Then 
			\begin{equation}\label{Eq2_Partition Domain}
				\mathbb{X}_{p}   =   B_{R} \cup \bigcup_{n=0}^{+ \infty } R^{n} \left(  \mathcal{V}  \right)  , 
			\end{equation}
			where $ \mathcal{V} $ and $B_{R}$ are as defined in relations \eqref{Eq2_Domain of the Contraction} and \eqref{Eq2_Contraction Domain B}, respectively.
			In particular, the sets $B_{R}$ and $R^{n} \left(  \mathcal{V}  \right) $, for $n \in \mathbb{N} $,  partition $\mathbb{X}_{p}$ and satisfy 
			\begin{equation}\label{Eq2_Partition Domain Explicit}
				R^{n} \left(  \mathcal{V}  \right)    =   R^{n} \left(  \mathbb{X}_{p}  \right)  \backslash R^{n+1} \left(  \mathbb{X}_{p}  \right) , \quad \text{for every }  n \in  \mathbb{N} .  
			\end{equation}
			
			Moreover, the restriction of $R$ to $B_{R}$, i.e., $R: B_{R} \to B_{R}$, is a homeomorphism, and $x_{R} \in B_{R}$. 
			
		\end{lemma}
		
		\paragraph{}
		
		To define the bijection $h_{1}$, observe that inequalities \eqref{Eq2_Bi-Lipschitz Contractions} and \eqref{Eq3_Contraction T Lower Bound} imply the existence of two continuous bijections $L: R \left(  \mathbb{X}_{p}  \right)  \to \mathbb{X}_{p}$ and $\Tilde{L}: T \left(  \mathbb{X}_{p}  \right)  \to \mathbb{X}_{p}$ that satisfy the relations $L \circ R = id$ and $ \Tilde{L} \circ T = id$, respectively. Therefore, Lemma \ref{Lem2_Domain of the Contraction} guarantees that the map $h_{1}: A_{T} \to A_{R}$, defined by the formula  
		\begin{equation}\label{Eq3_Conjugation Fundamental Domain A}
			h_{1} \left(  x  \right)   :=   R^{n} \left(  v  \right) , \quad \text{where } x = T^{n} \left(  v  \right)  \in A_{T}   \text{ for some } v \in \mathcal{V} \text{ and } n \in  \mathbb{N} ,   
		\end{equation} 
		is a well-defined bijection. 
		Moreover, $h_{1}$  satisfies the left-hand side of \eqref{Eq3_Conjugation h delta close to Identity}. Indeed, fix $ x \in A_{T}$, where $x = T^{n} \left(  v  \right) $ for some $v \in \mathcal{V}$ and $n \in \mathbb{N} $. Then, one has that $h_{1}(x) - x = R^{n}(v) - T^{n}(v)$. 
		If $n=0$, it is clear that $ \left| \left|  R^{0}(v) - T^{0}(v)  \right|  \right|_{p} = 0$. If $n \ge 1$, then the claim follows inductively from the inequality 
		\begin{align*}
			\left| \left|  R^{n}(v) - T^{n}(v)  \right|  \right|_{p}   & =    \left| \left|  R\circ R^{n-1} \left(  v  \right)  - R\circ T^{n-1} \left(  v  \right)  - \phi\circ T^{n-1} \left(  v  \right)   \right|  \right|_{p} \\ 
			& \underset{\eqref{Eq2_Bi-Lipschitz Contractions}}{\le}   \max \left\{  \delta, c_{2}\cdot  \left| \left|  R^{n-1}(v) - T^{n-1}(v)  \right|  \right|_{p}  \right\}  .
		\end{align*}
		Finally, observe that the map $h_{1}$ conjugates $R$ and $T$. Indeed, one has that  
		\begin{equation}\label{Eq3_Homeomorphism 1 Conjugation}
			R\circ h_{1}(x)    \underset{\eqref{Eq3_Conjugation Fundamental Domain A}}{=}    R^{n+1}(v)    \underset{\eqref{Eq3_Conjugation Fundamental Domain A}}{=}    h_{1}\circ T^{n+1}(v)    \underset{\eqref{Eq3_Conjugation Fundamental Domain A}}{=}    h_{1}\circ T(x) .
		\end{equation}

		\paragraph{} 
		
		To define the bijection $h_{2} : B_{T} \to B_{R} $, let 
		$$  \ell^{ \infty } \left(  \mathbb{Z} , \mathbb{X}_{p}  \right)  : =    \left\{   \left(  x_{n}   \right)_{n \in \mathbb{Z} } \in \mathbb{X}_{p}^{\mathbb{Z} } :   \sup_{n \in \mathbb{Z} }  \left| \left|  x_{n}  \right|  \right|_{p}   <   + \infty   \right\}  $$
		be the Banach space of bi-infinite bounded sequences in $\mathbb{X}_{p}$. 
		Endowed with the supremum norm $ \left| \left|  \cdot  \right|  \right|_{\ell^{ \infty } \left(  \mathbb{Z} , \mathbb{X}_{p}  \right) }$, the space $ \left(  \ell^{ \infty } \left(  \mathbb{Z} , \mathbb{X}_{p}  \right) ,  \left| \left|  \cdot  \right|  \right|_{\ell^{ \infty } \left(  \mathbb{Z} , \mathbb{X}_{p}  \right) }  \right) $ is complete. Notice that $\ell^{ \infty } \left(  \mathbb{Z} , \mathbb{Z}_{p}  \right) $ is the closed unit ball of the $\mathbb{Q}_{p}$-Banach space $\ell^{ \infty } \left(  \mathbb{Z} , \mathbb{Q}_{p}  \right) $ and, in particular, $\ell^{ \infty } \left(  \mathbb{Z} , \mathbb{Z}_{p}  \right)  = \mathbb{Z}_{p}^{\mathbb{Z} }$. 
		
		Fix $x \in B_{T}$. By Lemma \ref{Lem2_Domain of the Contraction}, the restriction $T: B_{T} \to B_{T}$ is a homeomorphism. Consequently, for all $x \in B_{T}$, the iterates $T^{n} \left(  x  \right)  \in B_{T}$ are well defined for all $n \in \mathbb{Z} $. This allows the definition of the map $W_{x}: \ell^{ \infty } \left(  \mathbb{Z} , \mathbb{X}_{p}  \right)  \to \ell^{ \infty } \left(  \mathbb{Z} , \mathbb{X}_{p}  \right) $ by  
		\begin{equation}\label{Eq3_Fixed Point Operator}
			W_{x} \left(   \left(  u_{n}   \right)_{n \in \mathbb{Z} }  \right)  :=    \left(  R \left(  T^{n-1}(x) + u_{n-1}   \right)  -  T^{n}(x)      \right)_{n \in \mathbb{Z} } .
		\end{equation} 
		This map is well-defined and is a contraction on $\ell^{ \infty } \left(  \mathbb{Z} , \mathbb{X}_{p}  \right) $, and therefore admits a unique fixed point $ \left(  z_{n} \left(  x  \right)    \right)_{n \in \mathbb{Z} }$, which in particular satisfies 
		\begin{equation}\label{Eq3_Fixed Point Recursive Relation} 
			R \left(  T^{n-1} \left(  x  \right)  + z_{n-1} \left(  x  \right)   \right)    =   T^{n} \left(  x  \right)  + z_{n} \left(  x  \right)  , \quad \text{for all } n \in \mathbb{Z} .
		\end{equation}

		Define $h_{2}: B_{T} \to B_{R}$ by  
		\begin{equation}\label{Eq2_Conjugation Fundamental Domain B}
			h_{2} \left(  x  \right)  :=   x + z_{0} \left(  x  \right)  
		\end{equation}
		To verify that $h_{2}$ is well defined, it remains to show that $h_{2} \left(  x  \right)  \in B_{R}$. To this end, relation \eqref{Eq3_Fixed Point Recursive Relation} implies that 
		\begin{equation}\label{Eq3_Inverse Images of B_R Points}
			T^{-n} \left(  x  \right)  + z_{-n} \left(  x  \right)    \in   R^{-n} \left\{  h_{2}  \left(  x  \right)   \right\}  \quad \text{for every } n \ge 0 .
		\end{equation} 
		This yields the claim immediately. Indeed, if $ h_{2}(x) \in A_{R}$, then there exist $v \in \mathcal{V}$ and $n_{0} \in \mathbb{N} $ such that $ h_{2}(x) = R^{n_{0}}(v)$. 
		In that case, $R^{-n_{0} -1 } \left\{   h_{2}(x)  \right\}  = \emptyset$, which contradicts~\eqref{Eq3_Inverse Images of B_R Points}.

		The map $h_{2}: B_{T} \to B_{R}$ is a bijection. To prove this, we define the map $\tilde{h}_{2}: B_{R} \to B_{T}$ in the same way as $h_{2}$, but with the roles of $R$ and $T$ reversed. 
		For every $x \in B_{R}$ and $y \in B_{T}$, the uniqueness of the points $h_{2} \left(  x  \right) $ and $\tilde{h}_{2} \left(  y  \right) $ implies that  $h_{2} \circ \tilde{h}_{2} \left(  y  \right)  = y$ and $\tilde{h}_{2} \circ h_{2} \left(  x  \right)   = x$. Thus, we conclude that $h_{2}$ is a bijection.

		To demonstrate that $h_{2}$ satisfies the right-hand side inequality of \eqref{Eq3_Conjugation h delta close to Identity}, we begin by observing the following:   
		\begin{align*}  \left| \left|  z_{n} \left(  x  \right)   \right|  \right|_{p}   &
			\underset{\eqref{Eq3_Fixed Point Operator}}{=}    \left| \left|  R \left(  T^{n-1} \left(  x  \right)  + z_{n-1} \left(  x  \right)   \right)  - T^{n} \left(  x  \right)    \right|  \right|_{p}  \\  
			& =    \left| \left|  R \left(  T^{n-1} \left(  x  \right)  + z_{n-1} \left(  x  \right)   \right)  - R \left(  T^{n-1} \left(  x  \right)   \right)  - \phi \left(  T^{n-1} \left(  x  \right)   \right)    \right|  \right|_{p}   \\  
			& \le    \max \left\{  c_{2}  \left| \left|  z_{n-1} \left(  x  \right)   \right|  \right|_{p},   \delta  \right\} ,
		\end{align*}
		for all $n \in \mathbb{Z} $. Since $ \left(  z_{n}   \right)_{n \in \mathbb{Z} } \in \ell^{ \infty } \left(  \mathbb{Z} , \mathbb{X}_{p}  \right) $, this recurrence relation implies that:  
		\begin{equation}\label{Eq3_Fixed Point delta small}
			\left| \left|  z_{n} \left(  x  \right)   \right|  \right|_{p}   \le   \delta, \quad \text{for every } x \in B_{T} \text{ and } n \in \mathbb{Z} . 
		\end{equation}
		This inequality ensures that $h_{2}$ satisfies the desired bound in \eqref{Eq3_Conjugation h delta close to Identity}.

		Finally, observe that the map $h_{2}$ conjugates $R$ and $T$. 
		Indeed, since the maps $W_{x}, W_{T \left(  x  \right) } : \ell^{ \infty }_{\mathbb{Z} } \left(  \mathbb{X}_{p}  \right)  \to \ell^{ \infty }_{\mathbb{Z} } \left(  \mathbb{X}_{p}  \right) $ are contractions, each admits a unique fixed point. Applying relation~\eqref{Eq3_Fixed Point Recursive Relation} to $x$ and $T \left(  x  \right) $, respectively, yields $ z_{n} \left(  x  \right)  = z_{n - 1} \left(  T \left(  x  \right)   \right) $, for all $n \in \mathbb{Z} $. Consequently,  
		\begin{align*} 
			R \circ h_{2} \left(  x  \right)    &\underset{\eqref{Eq2_Conjugation Fundamental Domain B}}{=}   R \left(  x + z_{0} \left(  x  \right)   \right)    \underset{\eqref{Eq3_Fixed Point Recursive Relation}}{=}   T \left(  x  \right)  + z_{1} \left(  x  \right)   \\  
			&=   T \left(  x  \right)  + z_{0} \left(  T \left(  x  \right)   \right)    =   h_{2} \circ T \left(  x  \right)  .
		\end{align*}
		Therefore, for every $x \in B_{T}$,  
		\begin{equation}\label{Eq3_Homeorphism 2 Conjugation}
			R\circ h_{2} \left(  x  \right)    =   h_{2} \circ T(x) .
		\end{equation}

		\paragraph{} 
		
		To prove that $R$ is strongly Lipschitz structurally stable, it remains to show that the conjugation map $h$ and its inverse $h^{-1}$ are continuous. 
		We prove the result for $h$ only, as the proof for $h^{-1}$ follows similarly. 
		
		\medskip
		
		First, it is shown that $h$ is continuous at every point $x \in A_{T}$. The set $A_{T}$ is open, being the union of the sets $T^{n} \left(  \mathcal{V}  \right) $ for $n \in \mathbb{N} $, where each $T^{n} \left(  \mathcal{V}  \right)  $ is open because $T$ is an open map and $\mathcal{V}$ is $\rho$-open. 
		Fix $x \in A_{T} $, and let $m \in \mathbb{N} $ and $v \in \mathcal{V}$ be such that $x = T^{m}(v)$. Choose $  \epsilon > 0$ sufficiently small so that the radius 
		$$  \eta :=   {c_{1}^{m} \over c_{2}^{m}} \cdot  \epsilon $$
		satisfies $B \left(  x,  \eta  \right)  \subseteq T^{m} \left(  \mathcal{V}  \right)  $. 
		For every $x' \in B \left(  x ,  \eta  \right) $, there exists $v' \in \mathcal{V}$ such that $x' = T^{m} \left(  v'  \right) $. Inequality~\eqref{Eq3_Contraction T Lower Bound} then yields  $ \left| \left|  v - v'  \right|  \right|_{p} \le \epsilon / c_{2}^{m}$, while inequality~\eqref{Eq2_Bi-Lipschitz Contractions} implies  
		$$  \left| \left|  h \left(  x  \right)  - h \left(  x'  \right)   \right|  \right|_{p}   =    \left| \left|  h_{1}(x) - h_{1} \left(  x' \right)   \right|  \right|_{p}   =    \left| \left|  R^{m} \left(  v  \right)  - R^{m} \left(  v'  \right)   \right|  \right|_{p}   \le    \epsilon . $$ 	
		Since $x \in A_{T}$ and $\epsilon  > 0$ were arbitrary, this proves that $h$ is continuous on $A_{T}$.

		\medskip
		
		The continuity of $h$ at points $x \in B_{T}$ is established separately in each of the cases considered in the theorem. More precisely, in the setting of Point~(1) the argument relies on the compactness of the space $\mathbb{Z}_{p}$, whereas in the setting of Point~(2) it uses the assumption that $R$ is scaling.

		\paragraph{Proof in the case of Point (1).} 
		
		Assume that $\mathbb{X}_{p} = \mathbb{Z}_{p}$. 
		Note that $B_{T} =  \left\{  x_{T}  \right\} $ and $B_{R} =  \left\{  x_{R}  \right\} $. To demonstrate this, assume there exists $x \in B_{T}$ such that $x \neq x_{T}$. By the construction of the set $B_{T}$, one has that $T \left(  B_{T}  \right)  = B_{T}$. Let $\textrm{diam} \left(  B_{T}  \right) $ be the diameter of $ B_{T}$, defined as   
		$$ \textrm{diam} \left(  B_{T}  \right)  :=   \sup \left\{   \left| \left|  x_{1} - x_{2}  \right|  \right|_{p} :   x_{1}, x_{2} \in B_{T}  \right\}    <   + \infty  .$$
		Assume $\textrm{diam} \left(  B_{T}  \right)   > 0$. Since $R$ is a contraction, one has that 
		$$ \textrm{diam} \left(  B_{T}  \right)    =   \textrm{diam} \left(  T \left(  B_{T}  \right)   \right)     \le   c_{2} \cdot \textrm{diam} \left(  B_{T}  \right)  ,$$
		which contradicts the assumption that $\textrm{diam} \left(  B_{T}  \right)   > 0$. Therefore, $\textrm{diam} \left(  B_{T}  \right)  =0$, and thus $ B_{T}  =  \left\{  x_{T}  \right\} $. 
		A similar argument applies for $B_{R}$, yielding $B_{R} =  \left\{  x_{R}  \right\} $.

		To conclude the proof, it remains to show that $h$ is continuous at $x_{T}$. 
		To this end, fix $ \eta  > 0$ and assume that $x \in \mathbb{Z}_{p}$ satisfies $ \left| \left|  x_{T} - x  \right|  \right|_{p} \le  \eta $ and $x \neq x_{T}$. Let $v = v \left(  x  \right)  \in \mathcal{V} $ and $n = n \left(  x  \right)  \in \mathbb{N} $ be such that $x = T^{n}(v)$. 
		It suffices to prove that $n \left(  x  \right)  \to + \infty $ as $x \to x_{T}$, since
		\begin{equation}\label{Eq2_Fixed Point Continuity}
			\left| \left|  h \left(  x_{T}  \right)  - h \left(  x  \right)   \right|  \right|_{p}   =    \left| \left|  x_{R} - R^{n}(v)  \right|  \right|_{p}    =    \left| \left|  R^{n} \left(  x_{R}  \right)  - R^{n}(v)  \right|  \right|_{p}   \underset{\eqref{Eq2_Bi-Lipschitz Contractions}}{\le}   c_{2}^{n} . 
		\end{equation}
		To this end, for every $k \in \mathbb{N} $, define $\eta_{k} := \inf  \left\{   \left| \left|  x_{T} - x'  \right|  \right|_{p} :   x' \in T^{k} \left(  \mathcal{V}  \right)   \right\} $.
		Due to the compactness of $\mathcal{V}$, one infers that $\eta_{k}  > 0$ for all $k \in \mathbb{N} $.
		Moreover, since $T$ is a contraction, the sequence $ \left\{  \eta_{k}  \right\}_{k \in  \mathbb{N} }$ is strictly decreasing, and thus $\eta_{k} {\longrightarrow} 0$ as $k \to + \infty $. 
		Let $k_{ \eta} \in \mathbb{N} $ be such that $\eta_{k_{ \eta} + 1}  <  \eta \le \eta_{k_{ \eta}} $. From the definition of the sequence $ \left(  \eta_{k}   \right)_{k \in  \mathbb{N} }$, one has that $n \ge k \left(   \eta  \right)  $. 
		Therefore, as $ \eta \to 0$, one obtains $\underset{ \eta \to 0}{\lim} k_{ \eta} = + \infty $, which implies that $\underset{ x \to x_{T} }{\lim} n(x) = + \infty $. 
		This shows that $h$ is continuous at $x_{T}$.

		\paragraph{Proof in the case of Point (2).} Assume that $\mathbb{X}_{p} = \mathbb{Q}_{p}$ and that $R$ is scaling. Recall that, by Lemma \ref{Lem2_Open bi-Lipschitz Contractions}, the map $R$ is open and its image $R \left(  \mathbb{X}_{p}  \right) $ is $\rho$-open for some $\rho  > 0$.

		\paragraph{} 	
		
		To prove that $h$ is continuous on $B_{T}$, we will use two claims. In the following, let $x \in B_{T}$ be fixed, and we will establish the continuity of $h$ at this point.
		
		\begin{claim}\label{Claim3 Homeomorphism 2 on B_T}
			For every $ \epsilon  > 0$, there exists $ 0 <  \eta <  \epsilon $ such that for all $x' \in B_{T}$ with $ \left| \left|  x - x'  \right|  \right|_{p} \le  \eta $, we have $  \left| \left|  h_{2} \left(  x  \right)   - h_{2} \left(  x'  \right)   \right|  \right|_{p}  \le  \epsilon $. 
			
		\end{claim}
		
		\begin{proof}[Proof of Claim \ref{Claim3 Homeomorphism 2 on B_T}]
			
			Fix $  \epsilon  > 0$ and $x' \in B_{T}$ such that $ \left| \left|  x - x'  \right|  \right|_{p} \le  \eta $, where $ 0 <  \eta <  \epsilon $ is a positive number that will be determined later. 
			To prove that $ \left| \left|  h_{2}(x) - h_{2} \left(  x'  \right)   \right|  \right|_{p} \le  \epsilon $, it suffices to show that $ \left| \left|  z_{0}(x) - z_{0} \left(  x'  \right)   \right|  \right|_{p} \le  \epsilon $, based on relation \eqref{Eq2_Conjugation Fundamental Domain B}.    
			Since $ \left(  z_{n} \left(  x  \right)    \right)_{n \in \mathbb{Z} } \in \ell^{ \infty } \left(  \mathbb{Z} , \mathbb{X}_{p}  \right) $ and $ \left(  z_{n} \left(  x'  \right)    \right)_{n \in \mathbb{Z} } \in \ell^{ \infty } \left(  \mathbb{Z} , \mathbb{X}_{p}  \right) $ are the fixed point of the maps $W_{x}$ and $W_{x'}$, respectively, from relation \eqref{Eq3_Fixed Point Operator}, one infers that for every $m \in \mathbb{Z} $ it holds that
			
			\begin{align*} 
				\left| \left|  z_{m} \left(  x  \right)  - z_{m} \left(  x'  \right)   \right|  \right|_{p}   \le    \max \left\{  \begin{matrix}  \left| \left|  T^{m} \left(  x  \right)  - T^{m} \left(  x'  \right)   \right|  \right|_{p},  \\    c_{2}   \left| \left|  T^{m-1} \left(  x  \right)  - T^{m-1} \left(  x'  \right)   \right|  \right|_{p},  \\ 
					c_{2}  \left| \left|  z_{m-1} \left(  x  \right)  - z_{m-1} \left(  x'  \right)   \right|  \right|_{p} \end{matrix}  \right\}  .
			\end{align*}
			
			By applying this inequality recursively for $m$ ranging from $0$ to $-n$, where $n \ge 0$ is a nonnegative integer, we obtain the following estimate: 
			\begin{align*}
				\left| \left|  z_{0} \left(  x  \right)  - z_{0} \left(  x'  \right)   \right|  \right|_{p}   & \le    \max \left\{  \begin{matrix}   \left| \left|  x - x'  \right|  \right|_{p} ,  \\  
					\max \left\{   c_{2}^{m}  \left| \left|  T^{-m} \left(  x  \right)  - T^{-m} \left(  x'  \right)   \right|  \right|_{p} :   1 \le m \le n  \right\}  ,   \\  
					c_{2}^{n}  \left| \left|  z_{-n} \left(  x  \right)  - z_{-n} \left(  x'  \right)   \right|  \right|_{p} \end{matrix}  \right\}    \\ 
			\end{align*}
			Upon taking into account inequality \eqref{Eq3_Fixed Point delta small}, the right-hand side of this inequality can be made smaller than $ \epsilon  > 0$. This can be done by first selecting $n \in \mathbb{N} $ sufficiently large and then choosing $ \eta  > 0$ small enough. 
			
			The proof of the claim is complete.

		\end{proof}

		\begin{claim}\label{Claim3 Homeomorphism 2 on A_T}
			Let $ \eta  > 0$ be a positive number, and  $T^{n} \left(  v  \right)  \in A_{T}$ for some $v \in \mathcal{V}$  be such that $ \left| \left|  x - T^{n} \left(  v  \right)   \right|  \right|_{p} \le  \eta $. Then, the following inequality holds:  
			$$   \left| \left|  h \left(  x  \right)   - h \left(  T^{n} \left(  v  \right)   \right)   \right|  \right|_{p}    \le     \max \left\{  \delta  c_{2}^{n_{ \eta}} ,    \eta  \right\}  ,$$
			where $n_{ \eta} \in \mathbb{N} $ is the largest non-negative integer such that $   \eta \cdot  \left(  1/ c_{1}  \right) ^{n_{ \eta}} \le \rho $. 
			
		\end{claim}
		
		\begin{proof}[Proof of Claim \ref{Claim3 Homeomorphism 2 on A_T}]
			Fix $  \eta  > 0$ and a point $T^{n} \left(  v  \right)  \in A_{T}$, where $v \in \mathcal{V}$ and $ \left| \left|  x - T^{n}(v)  \right|  \right|_{p} \le  \eta $. The choice of $n_{ \eta} $ guarantees that $n \ge n_{ \eta}$. Indeed, it ensures that  
			$$ \left| \left|  T^{-n_{ \eta}}(x) -  T^{- n_{ \eta}} \left(  T^{n}(v)  \right)   \right|  \right|_{p} \le \rho. $$
			Therefore since $T \left(  \mathbb{Q}_{p}  \right) $ is $\rho$-open, it follows that  $T^{n - n_{ \eta}}(v) \in T \left(  \mathbb{Q}_{p}  \right) $. In other words, the point $T^{n} \left(  v  \right) $ is at least $n_{ \eta}$-times invertible by the map $T$.

			By Lemma \ref{Lem2_Scaling Contraction plus Lipschitz}, one has that
			$$  \left| \left|  R^{n} \left(  T^{-n}(x)  \right)  - R^{n}(v)  \right|  \right|_{p}   \underset{\eqref{EqLem_Scaling Maps Induction}}{=}    \left| \left|  x - T^{n}(v)  \right|  \right|_{p} . $$ 
			Furthermore, it holds that: $  \left| \left|  h\circ T^{-n}(x)  - T^{-n}(x)  \right|  \right|_{p} \le \delta$ and $ R^{-n}\circ h(x)  = h\circ T^{-n} (x) $.
			Therefore,  one has that
			\begin{align*}
				\left| \left|  h(x) - h \left(  T^{n}(v)  \right)   \right|  \right|_{p}  & =    \left| \left|  R^{n}\circ R^{-n} \circ h(x) - R^{n}(v)  \right|  \right|_{p}  \\ 
				&=    \left| \left|  R^{n}\circ h \circ T^{-n}(x) - R^{n}\circ T^{-n} (x) + R^{n} \circ T^{-n}(x)   - R^{n}(v)  \right|  \right|_{p}  \\ 
				& \underset{\eqref{EqLem_Scaling Maps Induction}}{\le}   \max  \left\{  \delta  c_{2}^{n},    \left| \left|  x - T^{n}(v)  \right|  \right|_{p}  \right\}     \le   \max \left\{  \delta  c_{2}^{n_{ \eta}} ,    \eta  \right\} , 
			\end{align*}
			where the last inequality follows from the assumption. 
			This completes the proof of the claim. 
			
		\end{proof}

		From the Claims \ref{Claim3 Homeomorphism 2 on B_T} and \ref{Claim3 Homeomorphism 2 on A_T}, we have shown that $h$ is continuous at every point $x \in B_{T}$. Specifically, the first claim establishes a continuity condition on $h$ for small perturbations within $B_{T}$, while the second claim provides a bound on the continuity of $h$ when $x$ is close to a point in $A_{T}$. Therefore, $h$ is continuous on $B_{T}$.

		The proof of the theorem is complete.

	\end{proof}

	\paragraph{Acknowledgments}
	The second author was supported by the PROPE–UNESP Grant 013/2023. The third author was partially supported by CNPq under Grant 310784/2021-2 and by FAPESP Projects  2024/15612-6 and 2024/04685-2. The fourth author was supported by FAPESP Grants 2023/06371-2 and 2024/10135-5.
	
	The authors thank the anonymous referees for their careful reading and valuable comments, and for identifying an error in an earlier version of the manuscript. Special thanks are due to the second referee for providing the example of the contraction $R_{3}$ presented at the end of Section~\ref{Sec_Main Results}, as well as for suggesting an alternative proof of Lemma~\ref{Lem2_Open bi-Lipschitz Contractions}, which has been adapted in the final version.

	\vspace{1em}
	\noindent
	Danilo Antonio Caprio
	\newline
	UNESP -  Departamento de matem\'atica,
	Faculdade de Engenharia - Campus de Ilha Solteira.
	\newline Avenida Brasil, 56, Centro - Ilha Solteira, 15385-000, SP, Brasil.
	\newline
	e-mail: {\rm \texttt{danilo.caprio@unesp.br }}
	\newline
	\newline
	Fernando Lenarduzzi 
	\newline
	Faculdade ESEG - Grupo ETAPA. 
	\newline Rua Vergueiro, 1549, Vila Mariana, 04101-000, SP, Brasil. 
	\newline 
	e-mail: {\rm \texttt{matematica@etapa.com.br}}
	\newline 
	\newline
	Ali Messaoudi
	\newline
	UNESP - Departamento de matem\'atica, Instituto de Bioci\^encias Letras e Ci\^encias Exatas.
	\newline Rua Crist\'ov\~ao Colombo, 2265, Jardim Nazareth, S\~ao Jos\'e do Rio Preto, 15054-000, SP, Brasil.
	\newline
	e-mail: {\rm \texttt{ali.messaoudi@unesp.br}}
	\newline
	\newline
	Ioannis Tsokanos
	\newline
	UNESP - Departamento de matem\'atica, Instituto de Bioci\^encias Letras e Ci\^encias Exatas.
	\newline 
	Rua Crist\'ov\~ao Colombo, 2265, Jardim Nazareth, S\~ao Jos\'e do Rio Preto, 
	15054-000, SP, Brasil. 
	\newline
	e-mail: {\rm \texttt{ioannis.tsokanos@unesp.br }}


\end{document}